\documentclass[12pt]{article}

\usepackage{amsmath,amssymb,amsthm}

\usepackage{enumerate}

\usepackage{mathrsfs}

\usepackage[all]{xy}

\usepackage{amscd}

\usepackage{enumitem}

\usepackage{geometry}
\usepackage{layout}
\newtheoremstyle{mystyle} 
    {\topsep}				
    {\topsep}				
    {\normalfont}			
    {} 					
    {\bfseries} 			
    {\newline}				
    {5pt plus 1pt minus 1pt}{\underline{\thmname{#1}\thmnumber{#2}\thmnote{（#3）}}}	

\theoremstyle{definition}
\newtheorem{theorem}{Theorem}
\newtheorem{prop}[theorem]{Proposition}
\newtheorem{lem}[theorem]{Lemma}
\newtheorem{cor}[theorem]{Corollary}

\newtheorem{definition}[theorem]{Definition}
\newtheorem{rem}[theorem]{Remark}

\newtheorem{conj}[theorem]{Conjecture}

\numberwithin{theorem}{section}
\numberwithin{equation}{section}

\newtheorem*{introtheoremA}{Theorem A}
\newtheorem*{introcorollaryA}{Corollary A.1}
\newtheorem*{introcorollaryAtwo}{Corollary A.2}
\newtheorem*{introtheoremB}{Theorem B}

\newcommand{\ilim}{\mathop{\varinjlim}\limits}
\newcommand{\plim}{\mathop{\varprojlim}\limits}

\newcommand \defeq{\overset{\text{def}}=}
\newcommand \ab{\operatorname{ab}}
\newcommand \Gal{\operatorname{Gal}}
\newcommand \isom {\overset \sim \rightarrow}
\newcommand \longisom {\overset \sim \longrightarrow}

\newcommand \Ker{\operatorname{Ker}}

\def \Image{\operatorname{Im}}
\newcommand \res{\operatorname {res}}

\def \sol{\operatorname{sol}}
\def \Aut{\operatorname{Aut}}
\def \tor{\operatorname{tor}}
\def \inf{\operatorname{inf}}

\def\p{\frak{p}}
\def\q{\frak{q}}

\def\ur{\operatorname{ur}}

\def\Frob{\operatorname{Frob}}

\def\rank{\operatorname{rank}}

\def\Dec{\operatorname{Dec}}

\newcommand\bZ{\mathbb Z}

\newcommand\bQ{\mathbb Q}

\newcommand\bC{\mathbb C}
\newcommand\bF{\mathbb F}

\newcommand\ltilde{\tilde{l}}

\newcommand\bap{\overline{\p}}

\def \cs{\operatorname{cs}}
\def \Ram{\operatorname{Ram}}
\def\tr{\operatorname{tr}}
\def\sup{\operatorname{sup}}
\def\inf{\operatorname{inf}}
\def\res{\operatorname{res}}

\def\cC{{\mathcal C}}

\def \Iso{\operatorname{Iso}}
\def \OutIso{\operatorname{OutIso}}
\def \Inn{\operatorname{Inn}}
\def \Out{\operatorname{Out}}

\def\sol{\operatorname{sol}}
\def \N{\operatorname{N}}
\def\Syl{\operatorname{Syl}}

\def\Char{\operatorname{Char}}
\def\Deg{\operatorname{Deg}}
\def\max{\operatorname{max}}
\def\pr{\operatorname{pr}}

\begin{document}

\renewcommand{\thesection}{\arabic{section}}

\renewcommand\thefootnote{*\arabic{footnote}}

\title{The pro-$\mathcal{C}$ anabelian geometry of number fields}
\author{ \textsc{Ryoji Shimizu}\footnote{Institute of Science Tokyo, 2-12-1 Ookayama, Meguro-ku, Tokyo 152-8550, Japan.
e-mail: \texttt{shimizu.r.ap@m.titech.ac.jp}}}
\date{
}

\maketitle

\begin{abstract}
Let $K$ be a number field and $\cC$ a nontrivial full class
of finite groups.
Write $K^\cC/K$ for the maximal pro-$\cC$ Galois extension
of $K$ and $G_K^\cC$ for its Galois group.
Write $\Sigma(\cC)$ for the set of prime numbers $p$
such that $\bZ/p\bZ\in\cC$.
We prove that, for two such pairs $(K_i,\cC_i)$,
every isomorphism
$G_{K_1}^{\cC_1}\simeq G_{K_2}^{\cC_2}$
is induced by a unique isomorphism between
$K_2^{\cC_2}/K_2$ and $K_1^{\cC_1}/K_1$,
provided that $\Sigma(\cC_i)$ has positive upper Dirichlet
density for at least one $i$.
As a consequence, we generalize reconstruction results
of Uchida and Sueyoshi to Galois extensions $\Omega_i/K_i$
such that $\Omega_i$ has no nontrivial abelian
pro-$\Sigma$ extension, where $\Sigma$ has positive upper
Dirichlet density.
We also prove a relative version of the first result without a density assumption:
the same conclusion holds 
if the group isomorphism is compatible with the canonical restriction maps to the maximal abelian pro-$\Sigma(\cC_1)$ quotient of the absolute Galois group of $\bQ$,
and if $2\in\Sigma(\cC_1)$ or at least one of $K_1,K_2$ has a complex prime.
This relative result applies, in particular, to maximal
pro-$p$ quotients under these conditions.
\end{abstract}

\tableofcontents

\section{Introduction}

The Neukirch--Uchida theorem states that every isomorphism between the absolute Galois groups of number fields is induced by a unique isomorphism between their algebraic closures (cf. \cite{Neukirch,Uchida}).
Neukirch and Uchida also established a similar result
for maximal prosolvable quotients
(cf. \cite{Neukirch2,Uchida2}).
Further refinements have been obtained in several directions.
Ivanov \cite{Ivanov,Ivanov3} and the author
\cite{Shimizu,Shimizu2} studied reconstruction from Galois groups
with restricted ramification.
Sa\"\i di and Tamagawa \cite{Saidi-Tamagawa2} proved that
the isomorphism class of a number field is determined by
the maximal $3$-step solvable quotient of its absolute
Galois group, and obtained functorial reconstruction results
using higher solvable quotients.

In this paper, we study the reconstruction of number fields
from maximal pro-$\cC$ quotients of their absolute Galois groups.
Let $K$ be a number field and $\cC$ a nontrivial full class
of finite groups.
We write $K^\cC/K$ for the maximal pro-$\cC$ Galois extension
of $K$, and
\[
G_K^\cC \defeq G(K^\cC/K).
\]
Thus, $G_K^\cC$ is the maximal pro-$\cC$ quotient of the
absolute Galois group $G_K$.
Set
\[
\Sigma(\cC)
\defeq
\{p \text{ prime} \mid \bZ/p\bZ \in \cC\},
\]
and write $\delta_{\sup}$ for upper Dirichlet density.
Our first main result is the following.

\begin{introtheoremA}[Theorem \ref{6.1}]
For $i=1,2$, let $K_i$ be a number field and
$\cC_i$ a nontrivial full class of finite groups.
Assume that
\[
\delta_{\sup}(\Sigma(\cC_i))>0
\]
for at least one $i$.
Then the canonical map
\[
\Iso(K_2^{\cC_2}/K_2,K_1^{\cC_1}/K_1)
\longrightarrow
\Iso(G_{K_1}^{\cC_1},G_{K_2}^{\cC_2}),
\qquad
\tau\longmapsto
\bigl(g\longmapsto\tau^{-1}g\tau\bigr)
\]
is bijective.
In particular, if
$G_{K_1}^{\cC_1}\simeq G_{K_2}^{\cC_2}$,
then $K_1\simeq K_2$.
\end{introtheoremA}

Theorem A also yields the following more general
reconstruction theorem.

\begin{introcorollaryA}[Corollary \ref{solv.cl.}]
Let $\Sigma$ be a set of rational primes such that
$\delta_{\sup}(\Sigma)>0$.
For $i=1,2$, let $K_i$ be a number field and
$\Omega_i/K_i$ a Galois extension such that
$\Omega_i$ has no nontrivial abelian pro-$(\Sigma)$ extension.
Then the canonical map
\[
\Iso(\Omega_2/K_2,\Omega_1/K_1)
\longrightarrow
\Iso(G(\Omega_1/K_1),G(\Omega_2/K_2))
\]
is bijective.
\end{introcorollaryA}

When $\Sigma$ is the set of all rational primes,
Corollary A.1 recovers Uchida's reconstruction theorem
for solvably closed Galois extensions \cite{Uchida2}.
Sueyoshi \cite[Theorem 2]{Sueyoshi} proved a version of
Corollary A.1 under the stronger assumption that
$\Sigma$ has Dirichlet density one and the additional
assumption that $\mu_p\subset\Omega_i$ for every
$p\in\Sigma$ and $i=1,2$.
Corollary A.1 replaces the density-one assumption by
a positive upper Dirichlet density assumption 
and does not require $\mu_p\subset\Omega_i$.

\begin{introcorollaryAtwo}[Corollary \ref{6.6}]
Let $\cC$ be a nontrivial full class of finite groups.
Assume that $\delta_{\sup}(\Sigma(\cC))>0$.
Then, for any number fields $K_1,K_2$, the canonical map
\[
\Iso(K_2,K_1)
\longrightarrow
\OutIso(G_{K_1}^{\cC},G_{K_2}^{\cC})
\]
is bijective.
\end{introcorollaryAtwo}

Here, $\OutIso$ denotes the set of isomorphisms modulo
inner automorphisms.
Thus, the contravariant functor $K\mapsto G_K^\cC$ is fully faithful
for isomorphisms of number fields and outer isomorphisms
of profinite groups.
This applies, in particular, when $\cC$ is the class of
all finite $\Sigma$-groups for a set $\Sigma$ of positive
upper Dirichlet density.

Our second main result is a relative reconstruction theorem
which does not require a density assumption.
For a set $\Sigma$ of rational primes, write
$G_\bQ^{\ab,\Sigma}$ for the maximal abelian
pro-$\Sigma$ quotient of $G_\bQ$.
Recall that an isomorphism
$G_{K_1}^{\cC_1}\simeq G_{K_2}^{\cC_2}$
implies $\Sigma(\cC_1)=\Sigma(\cC_2)$
by Lemma \ref{2.0}.

\begin{introtheoremB}[Theorem \ref{relGC}]
For $i=1,2$, let $K_i$ be a number field and
$\cC_i$ a nontrivial full class of finite groups.
Let
\[
\sigma:G_{K_1}^{\cC_1}\isom G_{K_2}^{\cC_2}
\]
be an isomorphism.
Write $\Sigma$ for the common set $\Sigma(\cC_1)=\Sigma(\cC_2)$ (see Lemma \ref{2.0}).
Assume that $2\in\Sigma$ or that at least one of $K_1,K_2$
has a complex prime.
If the diagram
\[
\xymatrix{
G_{K_1}^{\cC_1}\ar[rr]^-{\sigma}_{\sim}\ar[rd]
&&G_{K_2}^{\cC_2}\ar[dl]\\
&G_\bQ^{\ab,\Sigma}
}
\]
commutes, where the diagonal arrows are the canonical
restriction maps, then there exists a unique
\[
\tau\in
\Iso(K_2^{\cC_2}/K_2,K_1^{\cC_1}/K_1)
\]
such that $\sigma(g)=\tau^{-1}g\tau$ for every
$g\in G_{K_1}^{\cC_1}$.
\end{introtheoremB}

In particular, Theorem B applies to maximal pro-$p$ quotients
if the above diagram commutes and either $p=2$ or
at least one of $K_1,K_2$ has a complex prime.

We now describe the methods and the organization of the paper.
In \S\S 1--3, 
we develop the local theory for maximal pro-$\cC$ quotients of absolute Galois groups.
After studying the structure of decomposition groups in \S 1,
we recover in \S 2 the decomposition groups at the primes
in $P_{K,f}^{\cC}$ (Corollary \ref{2.10}).
These are the nonarchimedean primes $\p$ such that
$\mu_l\subset K_\p^\cC$ for some $l\in\Sigma(\cC)$
different from the residue characteristic of $\p$
(Definition \ref{2.7}).
The argument uses class field theory, including the
Grunwald--Wang theorem, and Neukirch's Brauer group method
\cite[Theorem 1]{Neukirch}, together with techniques
developed by Ivanov \cite{Ivanov} and
Sa\"\i di--Tamagawa \cite{Saidi-Tamagawa2}.
We also recover the pro-$\cC$ parts
$\chi_{K,l}^{\cC}$ of the $l$-adic cyclotomic characters
for $l\in\Sigma(\cC)$ (Proposition \ref{cyc3}).
In \S 3, these results 
yield a unique one-to-one correspondence (called a local correspondence) between $P_{K_1,f}^{\cC_1}$ and $P_{K_2,f}^{\cC_2}$
for any isomorphism $G_{K_1}^{\cC_1}\isom G_{K_2}^{\cC_2}$ (Theorem \ref{3.5}).
No density assumption is needed for these results.

A new ingredient, developed in \S 4, is a separatedness
result for decomposition groups in the maximal pro-$\Sigma$
quotient 
\[
G\bigl(\bQ(\cup_{l\in\Sigma}\mu_l)/\bQ\bigr)^\Sigma.
\]
Using Kummer theory and the Chebotarev density theorem,
we show that, if $\delta_{\sup}(\Sigma)>0$, the
decomposition groups at distinct rational primes have an intersection that is open in neither group
(Proposition \ref{incomm4}).
Together with the cyclotomic characters recovered in \S 2,
this implies that the local correspondence preserves
residue characteristics (Proposition \ref{incomm5}).

In \S 5, we give two methods for obtaining field isomorphisms from isomorphisms of Galois groups.
The first, Proposition \ref{U4}, modifies Uchida's argument
in \cite{Uchida2} 
by using the recovered cyclotomic characters, 
thereby extending it to the setting of maximal pro-$\cC$ quotients.
This method applies when $\Sigma$ is infinite and
a local correspondence preserving residue characteristics
exists on a sufficiently large set of primes.
The second, Proposition \ref{4.7}, is based on the author's earlier work \cite{Shimizu,Shimizu2} and uses the maximal multiple $\bZ_l$-extension.
It requires a complex prime, but applies to local correspondences preserving residue characteristics and Frobenius elements on sets which have positive upper Dirichlet density after a suitable finite base change.

In \S 6, we show that $P_{K,f}^{\cC}$ has Dirichlet density one if $\sum_{p\in\Sigma(\cC)}p^{-1}=\infty$.
This provides the density condition needed to apply Proposition \ref{U4} in the proof of Theorem \ref{6.1}.

In \S 7, we combine these results to prove the main theorems and their corollaries.
We conclude with two conjectures concerning reconstruction without density assumptions, and discuss their relation to the work of Koymans and Pagano \cite{Koymans-Pagano} in Remark \ref{7.5}~$(iii)$.

\section*{Acknowledgements}
The author would like to thank Professor Akio Tamagawa for helpful advice and carefully reading preliminary versions of the present paper. 
The author would also like to thank Peter Koymans and Carlo Pagano for bringing their paper \cite{Koymans-Pagano} to the author's attention after the first version of the present paper was posted on arXiv. The author then added Remark \ref{7.5}~$(iii)$.
This work was supported by JSPS KAKENHI Grant Numbers 21J11879 and 25KJ0125.

\section*{AI Use Declaration}
The first version of this paper was written entirely by the author without the use of AI tools and was posted on arXiv in March 2023.

During the revision of the manuscript, the author used OpenAI's ChatGPT to assist with proofreading, improving the English, and clarifying the mathematical exposition. The only contribution of ChatGPT to the mathematical content of the paper was its suggestion to use the Brun--Titchmarsh inequality in Remark 6.2 to justify the density formula.
The author takes full responsibility for the mathematical content and correctness of the paper.

\section*{Notation}
\begin{itemize}[leftmargin=*]
\item[$\bullet$]
Given a set $A$ we write $\# A$ for its cardinality.

\item
A number field is a finite extension 
 of the field of rational numbers $\bQ$. 
For a possibly infinite algebraic extension $F$ of $\bQ$, 
we write 
$
P_F$ 
for the set of primes 
of $F$, 
$
P_{F,\infty}$ for the set of archimedean primes of $F$, 
and 
$
P_{F,p}$ for the set of nonarchimedean primes of $F$ above 
a prime number $p$.
Further, for a set of primes $S \subset P_F$, 
we set $S_f\defeq S \setminus P_{F,\infty}$, 
$\underline{S}
\defeq \{ p \in P_{\bQ} \mid P_{F,p} \subset S \}$.
For $\bQ \subset F \subset F' \subset \overline{\bQ}$, 
we write 
$S(F')$ for the set of primes of $F'$ above the primes in $S$: 
$S(F') \defeq \{ \p \in P_{F'} \mid \p|_F \in S \}$.

\item[$\bullet$]
Let $G$ be a profinite group.
Write $G^{\sol}$ for the maximal prosolvable quotient of $G$.
Let $\overline{[G, G]}$ be the closed subgroup of $G$ which is (topologically) generated by the commutators in $G$. We write $G^{\ab} \defeq G/\overline{[G,G]}$
for the maximal abelian quotient of $G$.
Let $\cC$ be a (non-empty) full class of finite groups (cf. \cite[(3.5.2) Definition]{NSW}).
Write $G^\cC$ for the maximal pro-$\cC$ quotient of $G$, 
and $\Sigma(\cC)$ for the set of prime numbers $p$ with $\bZ/p\bZ \in \cC$.
We say that $\cC$ is nontrivial if $\Sigma(\cC) \not= \emptyset$.
Let $\Sigma \subset P_{\bQ,f}$. Write $(\Sigma)$ for the full class of finite groups whose orders are divisible only by prime numbers in $\Sigma$, and 
$G^\Sigma \defeq G^{(\Sigma)}$.
For a prime number $l$, 
we write $G^{(l)} \defeq G^{\{ l \}}$ for the maximal
pro-$l$ quotient of $G$.

\item
Given a Galois extension $L/K$, we write $G(L/K)$ for its Galois group $\Gal(L/K)$.
Let $K$ be a field.
We write $\overline{K}$ for a separable closure of $K$, 
and $G_K$ for the absolute Galois group $G(\overline{K}/K)$ of $K$.
For 
a full class of finite groups $\cC$, 
we write 
$K^{\cC}$ 
for the maximal 
pro-$\cC$ 
Galois extension of $K$, 
which corresponds to the quotient 
$G_K \twoheadrightarrow G_K^{\cC}$.
Note that $L^{\cC}=K^{\cC}$ for any subextension $L$ of $K^{\cC}/K$.
For $\Sigma \subset P_{\bQ,f}$ (resp. $l \in P_{\bQ,f}$), write $K^{\Sigma} \defeq K^{(\Sigma)}$ (resp. $K^{(l)} \defeq K^{(\{ l \})}$).

\item
For $i=1,2$, 
let $A_i$ be a (commutative) ring.
Write $\Iso(A_2, A_1)$ for the set of ring isomorphisms from $A_2$ to $A_1$.
For $i=1,2$, 
let $B_i$ be a ring containing $A_i$.
Write $\Iso(B_2/A_2, B_1/A_1) \defeq \{ \tau\in \Iso(B_2, B_1) \mid \tau(A_2)=A_1 \}$.

\item
For a profinite group $G$, write $\Inn(G)$ for the set of inner automorphisms of $G$.
For $i=1,2$, 
let $G_i$ be a profinite group.
Write $\Iso(G_1, G_2)$ for the set of isomorphisms of profinite groups from $G_1$ to $G_2$.
Note that the group $\Inn(G_2)$ acts on $\Iso(G_1, G_2)$ by the rule $\sigma(\phi) \defeq \sigma\circ\phi$ for $\sigma\in\Inn(G_2)$ and $\phi\in\Iso(G_1, G_2)$.
We call 
$\OutIso(G_1, G_2) \defeq \Iso(G_1, G_2)/\Inn(G_2)$
the set of outer isomorphisms from $G_1$ to $G_2$.
Write $\Out(G) \defeq \OutIso(G, G)$.

\item
Given an algebraic extension $K$ of $\bQ$ and $\p \in P_{K,f}$, 
we write $\kappa(\p)$ for the residue field at $\p$. 
When $K$ is a number field, we write $K_\p$ for the completion of $K$ at $\p$, 
and, in general, we write $K_\p$ for the union of $K'_{\p|_{K'}}$ for finite subextensions $K'/\bQ$ of $K/\bQ$.

\item
Let $L/K$ be a finite extension of number fields and $\q \in P_{L,f}$, 
and set $\p = \q|_{K}$.
We write $f_{\q,L/K} \defeq [\kappa(\q):\kappa(\p)]$.
We write $\cs(L/K)$ 
(resp. $\Ram(L/K)$) 
for the set of nonarchimedean primes of $K$
which split completely 
(resp. are ramified) 
in $L/K$.

\item
Let $K$ be a number field and $\p \in P_{K,f}$, 
and set $p = \p|_{\bQ}$.
Define the residue degree (resp. the local degree) of $\p$ 
as 
 $f_{\p,K/\bQ}$ 
(resp. $[K_\p:\bQ_p]$).
We set $\frak{N}(\p) \defeq \# \kappa(\p) = p^{f_{\p,K/\bQ}}$.

\item
For a number field $K$ and a set of primes $S \subset P_K$, 
we set 
$$
\delta_{\sup}(S) \defeq  
\limsup_{s \to 1+0} \frac{\sum_{\p \in S_f} \frak{N}(\p)^{-s}}{\log{\frac{1}{s-1}}}
,\ 
\delta_{\inf}(S) \defeq  
\liminf_{s \to 1+0} \frac{\sum_{\p \in S_f} \frak{N}(\p)^{-s}}{\log{\frac{1}{s-1}}}
$$
and 
if $\delta_{\sup}(S) = \delta_{\inf}(S)$, 
then write $\delta(S)$ (the Dirichlet density of $S$) for them.

\item
For $\bQ \subset F \subset F' \subset 
\overline{\bQ}$ with $F'/F$ Galois, $\q \in P_{F',f}$ and $\p = \q|_F$, 
write $
D_{\q}(F'/F) \subset G(F'/F)$ for the decomposition group (i.e. the stabilizer) of $G(F'/F)$ at $\q$. 
We sometimes write $D_{\q} = D_{\q}(F'/F)$, when no confusion arises.
Further, we also write $D_{\p} 
= D_{\p}(F'/F) = D_{\q}(F'/F)$, when no confusion arises. Note that $D_{\p}$ is only defined up to conjugation.
There exists a canonical isomorphism $D_{\q}(F'/F) \simeq G(F'_\q/F_\p)$,
and 
we will identify $D_{\q}(F'/F)$ with $G(F'_\q/F_\p)$ via this isomorphism.
Write $I_{\q} = I_{\q}(F'/F) = I_{\p} = I_{\p}(F'/F)$ for the inertia subgroup of $D_{\q}$.
For $S' \subset P_{F'}$, write 
\begin{equation*}
\begin{split}
&\Dec(F'/F, S') \defeq 
\left\{ D\subset G(F'/F) \left|
\begin{array}{l}
\text{$D= D_{\q}(F'/F)$ for some $\q \in S'$}
\end{array}
\right.\right\}.
\end{split}
\end{equation*}
Note that for $S \subset P_{F}$, $G(F'/F)$ acts on $\Dec(F'/F, S(F'))$ by conjugation.

\item
Let $p$ be a prime number.
A 
$p$-adic field is a finite extension 
 of the field of $p$-adic numbers $\bQ_p$. 
Let $\kappa$ be 
a $p$-adic field.
We write $V_{\kappa}$ 
(resp. $I_{\kappa}$) 
for the ramification 
(resp. inertia) 
subgroup of $G_\kappa$, and set 
$G_\kappa^{\tr} \defeq G_\kappa/V_{\kappa}$
 (resp. $G_\kappa^{\ur} \defeq G_\kappa/I_{\kappa}$). 

\item
For an abelian group $A$, we write $A_{\tor}$ for the torsion subgroup of $A$.
For an abelian profinite group $A$, we write $\overline{A_{\tor}}$ for the closure in $A$ of $A_{\tor}$, and set $A^{/\tor} \defeq A/\overline{A_{\tor}}$.

\item
For $n \in \bZ_{>0}$ not divisible by the characteristic of $K$, we write $\mu_n=\mu_n (\overline{K}) \subset \overline{K}$ for the $n$-th roots of unity. Let $l$ be a prime number distinct from the characteristic of $K$.
We set $\mu_{l^\infty} \defeq \bigcup_{n \in \bZ_{>0}} \mu_{l^n}(\overline{K}) \subset \overline{K}$.
Write $\chi_{K,l}$ for the $l$-adic cyclotomic character $G_{K}\rightarrow \Aut(\mu_{l^\infty})={\bZ_l}^{\ast}$, and for a full class of finite groups $\cC$, $\chi_{K,l}^{\cC} \colon G_{K}^{\cC}\rightarrow {\bZ_l}^{\ast, \cC}$ for the pro-$\cC$ factor of $\chi_{K,l}$.

\end{itemize}

\section{Fundamental properties of the decomposition groups
}
In the rest of this paper, let $K$ be a number field and 
$\cC$ a nontrivial full class of finite groups.

We begin by collecting results on the structures of the decomposition groups of $G_K^{\cC}$.
Write $K^{\cC-\ab}/K$ for the maximal pro-$\cC$ abelian extension, corresponding to $G_K^{\cC,\ab} =G_K^{\ab,\cC}$.
Let $\p \in P_{K,f}$.
Then there exists a canonical surjection $G_{K_\p}^{\cC,\ab} \twoheadrightarrow D_{\p}(K^{\cC-\ab}/K)$.
By the Grunwald-Wang theorem obtained by classical class field theory, we can prove the following assertions: 
\begin{lem}\label{1.1}
Let $\p \in P_{K,f}$.
Then 
the canonical surjection $G_{K_\p}^{\cC,\ab} \twoheadrightarrow D_{\p}(K^{\cC-\ab}/K)$ is bijective.
\end{lem}
\begin{proof}
By the Grunwald–Wang theorem, 
all local abelian extensions are realized by global abelian extensions, 
in other words, 
the composite $G_{K_\p}^{\ab} \twoheadrightarrow D_{\p}(K^{\ab}/K) \hookrightarrow G_K^{\ab}$ is injective 
(cf. \cite[III, 4.5 Theorem]{Gras}).
The assertion follows by passing to the pro-$\Sigma(\cC)$ factors of these abelian profinite groups.
\end{proof}

\begin{prop}\label{1.2}
Let $\overline{\p} \in P_{K^{\cC},f}$, and write $\p = \overline{\p}|_K$. 
Then 
the canonical surjection $G_{K_\p}^{\cC} \twoheadrightarrow D_{\overline{\p}}(K^{\cC}/K)$ is bijective.
\end{prop}
\begin{proof}
Since $G_{K_\p}$ is prosolvable, so is $G_{K_\p}^{\cC}$. Therefore, applying Lemma \ref{1.1} to finite extensions of $K$ corresponding to various open subgroups of $G_K^{\cC}$, we obtain the assertion.
\end{proof}

\begin{lem}\label{1.3}
Let $p$ be a prime number and $\kappa$ a $p$-adic field. 
Then the cohomological dimension of $G_{\kappa}^{\cC}$ 
is at most $2$.
In particular, $G_{\kappa}^{\cC}$ does not have a nontrivial torsion element.
\end{lem}
\begin{proof}
For $l\not\in \Sigma(\cC)$, the cohomological $l$-dimension of $G_{\kappa}^{\cC}$ is $0$ by \cite[(3.3.7) Corollary]{NSW}.
For $l\in \Sigma(\cC)$, by \cite[(7.5.8) Proposition]{NSW}, the inflation map $H^i(G_{\kappa}^{\cC},A)(l) \to H^i(G_{\kappa},A)(l)$ is bijective for any 
discrete torsion 
$G_{\kappa}^{\cC}$-module $A$, 
and hence the cohomological $l$-dimension of $G_{\kappa}^{\cC}$ is at most that of $G_{\kappa}$, which is $2$ by \cite[(7.1.8) Theorem (i)]{NSW}.
The last assertion follows from the fact that profinite groups of finite cohomological dimension do not contain nontrivial torsion elements, as explained in the argument following \cite[(3.3.1) Definition]{NSW}.
\end{proof}

\begin{prop}\label{1.4}
Let $\overline{\p}, \overline{\q} \in P_{K^{\cC},f}$. 
Then 
$D_{\overline{\p}}(K^{\cC}/K) \cap D_{\overline{\q}}(K^{\cC}/K) = 1$ 
if $\overline{\p} \neq \overline{\q}$.
In particular, $D_{\overline{\p}}(K^{\cC}/K) = D_{\overline{\q}}(K^{\cC}/K)$ if and only if $\overline{\p} = \overline{\q}$.
\end{prop}
\begin{proof}
Write $\p \defeq \overline{\p}|_K$ and $\q \defeq \overline{\q}|_K$.
Assume that $\overline{\p} \neq \overline{\q}$ and that $D_{\overline{\p}}(K^{\cC}/K) \cap D_{\overline{\q}}(K^{\cC}/K) \neq 1$.
Then, by Proposition \ref{1.2} and Lemma \ref{1.3}, 
$D_{\overline{\p}}(K^{\cC}/L) \cap D_{\overline{\q}}(K^{\cC}/L) \neq 1$ for any finite subextension $L$ of $K^{\cC}/K$.
Therefore, replacing $G_K^{\cC}$ by an open subgroup of $G_K^{\cC}$, we may assume $\p \neq \q$.

Take $1 \neq x \in D_{\overline{\p}}(K^{\cC}/K) \cap D_{\overline{\q}}(K^{\cC}/K)$, 
and an open normal subgroup $U$ of $G_K^{\cC}$ such that the image of $x$ in $G_K^{\cC}/U$ is not trivial.
Let $H$ be an open subgroup of $G_K^{\cC}$ topologically generated by $U$ and $x$. 
Then the image of $x$ in $H^{\ab}$ is not trivial.
However, by \cite[III, 4.16.7 Corollary]{Gras}, the intersection of the decomposition groups of $H^{\ab}$ at $\p$ and $\q$ is trivial, a contradiction.

The second assertion follows immediately from the first and Proposition \ref{1.2}.
\end{proof}

\begin{prop}\label{1.5}
Let $\tau \in \Aut(K^{\cC})$. 
Assume $\tau(K)=K$ and that the automorphism of $G_{K}^{\cC}$ induced by the conjugation action of $\tau$ 
is trivial.
Then $\tau$ is trivial.
\end{prop}
\begin{proof}
Write $K_0$ for the $\Aut(K^{\cC})$-invariant subfield of $K^{\cC}$. 
Then $K^{\cC}/K_0$ is Galois.
Let $\overline{\p} \in P_{K^{\cC},f}$. Then 
we have 
$D_{\tau\overline{\p}}(K^{\cC}/K) 
= \tau^{-1} D_{\overline{\p}}(K^{\cC}/K) \tau
= D_{\overline{\p}}(K^{\cC}/K)$ in $\Aut(K^{\cC})$.
Therefore, by Proposition \ref{1.4}, we obtain $\tau\overline{\p}=\overline{\p}$, 
and hence $\tau \in D_{\overline{\p}}(K^{\cC}/K_0)$.
Thus, $\tau \in 
\cap_{\overline{\p} \in P_{K^{\cC},f}} D_{\overline{\p}}(K^{\cC}/K_0)$.

By the Chebotarev density theorem, $\# \cs(K/K_0) = \infty$.
Take two distinct primes $\overline{\q},\overline{\q}' \in \cs(K/K_0)(K^{\cC})$.
Then 
$$\cap_{\overline{\p} \in P_{K^{\cC},f}} D_{\overline{\p}}(K^{\cC}/K_0)
\subset D_{\overline{\q}}(K^{\cC}/K_0) \cap D_{\overline{\q}'}(K^{\cC}/K_0) = D_{\overline{\q}}(K^{\cC}/K) \cap D_{\overline{\q}'}(K^{\cC}/K) = 1$$ by Proposition \ref{1.4}.
Thus, $\tau$ is trivial.
\end{proof}

\begin{prop}\label{1.6}
The profinite group $G_{K}^{\cC}$ is slim, i.e. for every open subgroup $H \subset G_{K}^{\cC}$, the centralizer $Z_{G_{K}^{\cC}}(H)$ is trivial.
In particular, $G_{K}^{\cC}$ has a trivial center (cf. \cite[Remark 0.1.3]{Mochizuki}).
\end{prop}
\begin{proof}
Let $H \subset G_{K}^{\cC}$ be any open subgroup. To prove the triviality of $Z_{G_{K}^{\cC}}(H)$, replacing $H$ by an open normal subgroup of $G_{K}^{\cC}$ contained in $H$, we may assume that $H \subset G_{K}^{\cC}$ is normal.
Let $\tau \in Z_{G_{K}^{\cC}}(H)$. Then the automorphism of $H$ induced by the conjugation action of $\tau$ is trivial. Therefore, $\tau$ is trivial by Proposition \ref{1.5}, and hence we obtain $Z_{G_{K}^{\cC}}(H)=1$.
In particular, we obtain $Z_{G_{K}^{\cC}}(G_{K}^{\cC}) =1$.
\end{proof}

\section{Recovering the decomposition groups
}
In this section, we recover group-theoretically a certain subset of the set of decomposition groups 
of $G_K^{\cC}$.

\begin{lem}\label{2.0}
The set $\Sigma(\cC)$ can be recovered group-theoretically from $G_{K}^{\cC}$.
\end{lem}
\begin{proof}
It follows from the theory of cyclotomic extensions that $G_{K}^{\cC}$ admits a $\bZ/p\bZ$-quotient for every $p\in \Sigma(\cC)$.
By the definition of $\Sigma(\cC)$, the assertion follows.
\end{proof}

\begin{definition}\label{2.1}
Let $l$ be a prime number.
Given a profinite group $G$, we say that $G$ is of $l$-decomposition type if $G$ is isomorphic to 
a semidirect product $\bZ_l \rtimes_\phi \bZ_l$ where $\phi : \bZ_l \to \Aut(\bZ_l)
$ is injective.
\end{definition}

If $l \neq 2$ (resp. $l=2$), the condition that $\phi : \bZ_l \to \Aut(\bZ_l) (= \bZ_l^{\ast} \simeq \bZ/(l-1)\bZ \times \bZ_l$ (resp. $\bZ/2\bZ \times \bZ_2))$ is injective is equivalent to the condition that $\Image(\phi)$ is nontrivial (resp. $\# \Image(\phi) \geq 3$).
Note that a profinite group $G$ of $l$-decomposition type is a pro-$l$ Demushkin group of rank $2$ (cf. \cite[Definition 2.1]{Ivanov}), 
in particular, 
$H^2(G,\bF_l) \simeq \bF_l$.

\begin{lem}\label{2.2}
Let $p,l$ be distinct prime numbers and $\kappa$ a $p$-adic field. 
Assume $l \in \Sigma(\cC)$.
If $\mu_{l} \subset \kappa^{\cC}$ (resp. $\mu_{l} \not\subset \kappa^{\cC}$), then 
the $l$-Sylow subgroups of $G_{\kappa}^{\cC}$ 
are of $l$-decomposition type (resp. isomorphic to $\bZ_l$).
\end{lem}
\begin{proof}
By \cite[(7.5.2) Proposition]{NSW}, 
we obtain the following split exact sequence, originally due to Iwasawa: 
\begin{equation}\label{tr}
1 \to (\prod_{q\neq p}\bZ_{q})(1) \to G_\kappa^{\tr} \to G_\kappa^{\ur} (\simeq \hat{\bZ}) \to 1,
\end{equation}
so that the $l$-Sylow subgroups of $G_{\kappa}^{\tr}$ 
are of $l$-decomposition type. 
Let $G_{\kappa,l}$ be any $l$-Sylow subgroup of $G_{\kappa}$.
Since $V_{\kappa}$
is pro-$p$, $G_{\kappa,l}$ 
is isomorphic to an $l$-Sylow subgroup of $G_{\kappa}^{\tr}$, and hence is of $l$-decomposition type.

Assume $\mu_{l} \subset \kappa^{\cC}$.
Let $\pi$ be a uniformizer of $\kappa$.
Then the composite 
$$G_{\kappa,l} \hookrightarrow G_{\kappa} \twoheadrightarrow G_{\kappa}^{\cC} \twoheadrightarrow G(\kappa(\mu_{l^{\infty}})(\sqrt[l^{\infty}]{\pi})/\kappa)$$
is injective, so that $G_{\kappa,l}$ is isomorphic to an $l$-Sylow subgroup of $G_{\kappa}^{\cC}$.
Thus, the $l$-Sylow subgroups of $G_{\kappa}^{\cC}$ are of $l$-decomposition type.

Assume $\mu_{l} \not\subset \kappa^{\cC}$.
Then, by the exact sequence (\ref{tr}), 
any $l$-Sylow subgroup of the inertia subgroup of $G_{\kappa}^{\cC}$ is trivial.
Thus, the $l$-Sylow subgroups of $G_{\kappa}^{\cC}$ are isomorphic to $\bZ_l$.
\end{proof}

\begin{lem}\label{2.3}
Let $p$ be a prime number and $\kappa$ a $p$-adic field. 
Then $G_{\kappa}^{\cC}$ does not have a closed subgroup of $p$-decomposition type.
\end{lem}
\begin{proof}
If $p\notin\Sigma(\cC)$, then $G_\kappa^{\cC}$ has no
nontrivial pro-$p$ subgroup, so the assertion is immediate.
Assume that $p\in\Sigma(\cC)$.
For any finite subextension $\kappa^{\cC}/\lambda/\kappa$, we have $\kappa^{\cC}\supset\lambda^{(p)}$.
Therefore, $G_{\kappa}^{\cC}$ does not contain any closed subgroup of $p$-decomposition type (see \cite[Lemma 3.3]{Ivanov3}, whose proof relies on arguments in the proof of \cite[Lemma 3.2]{Ivanov}).
\end{proof}

\begin{lem}\label{2.4}
Let $\overline{\p} \in P_{K^{\cC},f}$, and $H \subset D_{\overline{\p}}(K^{\cC}/K)$ be a nontrivial subgroup. 
Then the normalizer subgroup $\N_{G_{K}^{\cC}}(H)$ of $H$ in $G_{K}^{\cC}$ 
is contained in $D_{\overline{\p}}(K^{\cC}/K)$.
\end{lem}
\begin{proof}
Take $x \in \N_{G_{K}^{\cC}}(H)$.
Then $H = x^{-1}Hx \subset x^{-1}D_{\overline{\p}}(K^{\cC}/K)x = D_{x\overline{\p}}(K^{\cC}/K)$.
Thus, $H \subset D_{\overline{\p}}(K^{\cC}/K) \cap D_{x\overline{\p}}(K^{\cC}/K)$.
By Proposition \ref{1.4}, $\overline{\p} = x\overline{\p}$, 
so that $x \in D_{\overline{\p}}(K^{\cC}/K)$.
\end{proof}

\begin{lem}\label{2.5}
Let $\overline{\p} \in P_{K^{\cC},f}$, and $H \subset G_{K}^{\cC}$ be a closed subgroup of infinite order.
Assume that there exists an open subgroup of $H$ contained in $D_{\overline{\p}}(K^{\cC}/K)$.
Then $H \subset D_{\overline{\p}}(K^{\cC}/K)$.
\end{lem}
\begin{proof}
Let $H_0$ be an open subgroup of $H$  with $H_0 \subset D_{\overline{\p}}(K^{\cC}/K)$.
Replacing $H_0$ by the intersection of all conjugates of $H_0$ in $H$, we may assume that $H_0$ is an open normal subgroup of $H$.
Then, by Lemma \ref{2.4}, we have $H \subset \N_{G_{K}^{\cC}}(H_0) \subset D_{\overline{\p}}(K^{\cC}/K)$.
\end{proof}

Let $l \in \Sigma(\cC)$.
By \cite[(10.4.8) Corollary and (7.5.8) Proposition]{NSW}, 
the inflation maps $H^2(G_{K}^{\cC},\bF_l) \to H^2(G_{K},\bF_l)$ and $H^2(G_{K_{\p}}^{\cC},\bF_l) \to H^2(G_{K_{\p}},\bF_l)$ for $\p \in P_K$ 
are bijective.
Therefore, 
the canonical homomorphism $H^2(G_{K},\bF_l) \to \bigoplus_{\p \in P_K} H^2(G_{K_{\p}},\bF_l)$ (cf. \cite[(8.6.1) Proposition]{NSW}) may be naturally identified with the canonical homomorphism $H^2(G_{K}^{\cC},\bF_l) \to \bigoplus_{\p \in P_K} H^2(D_{\p}(K^{\cC}/K),\bF_l)$ by Proposition \ref{1.2}.

\begin{lem}\label{2.6}
Let $l \in \Sigma(\cC)$.
Then the canonical homomorphism 
$$H^2(G_{K}^{\cC},\bF_l) \to \bigoplus_{\p \in P_K} H^2(D_{\p}(K^{\cC}/K),\bF_l)$$
is injective.
\end{lem}
\begin{proof}
By the natural identification as above, we may reduce to the case where $\cC$ is the full class of all finite groups, in which case $G_{K}^{\cC} = G_{K}$.
Then the injectivity in question can be deduced from the Hasse principle (see \cite[(9.1.10) Corollary (i)]{NSW}), or alternatively, it may be verified directly as follows.
Since $[K(\mu_l):K]$ is prime to $l$, the restriction map 
$H^2(G_{K},\bF_l) \to H^2(G_{K(\mu_l)},\bF_l)$ 
is injective.
Therefore, we may reduce to the case where $\mu_l \subset K$.
Then $H^2(G_{K},\bF_l)$, $H^2(G_{K_{\p}},\bF_l)$ are canonically isomorphic to the $l$-torsion subgroups of the Brauer groups $Br(K)$, $Br(K_\p)$ of $K$, $K_\p$, respectively (cf. \cite[(6.3.6) Corollary]{NSW}).
Thus, the assertion follows from the injectivity of the canonical homomorphism $Br(K) \to \bigoplus_{\p \in P_K} Br(K_\p)$ (see \cite[(8.1.17) Theorem]{NSW}).
\end{proof}

\begin{definition}\label{2.7}
For $l \in \Sigma(\cC)$, set 
$$P_{K,f}^{\cC,l} \defeq \{ \p\in P_{K,f} \setminus P_{K,l} \mid \mu_{l} \subset K_{\p}^{\cC} \},\ \Dec_{K,f}^{\cC,l} \defeq 
\Dec(K^{\cC}/K, P_{K,f}^{\cC,l}(K^{\cC})).$$
Set 
$$P_{K,f}^{\cC} \defeq \cup_{l \in \Sigma(\cC)}P_{K,f}^{\cC,l},\ \Dec_{K,f}^{\cC} \defeq \cup_{l \in \Sigma(\cC)}\Dec_{K,f}^{\cC,l}=\Dec(K^{\cC}/K, P_{K,f}^{\cC}(K^{\cC})).$$
\end{definition}

\begin{rem}\label{csset}
$(i)$ 
Let $L$ be a finite subextension of $K^{\cC}/K$.
We have $P_{L,f}^{\cC,l} = P_{K,f}^{\cC,l}(L)$ for $l \in \Sigma(\cC)$ 
and $P_{L,f}^{\cC} = P_{K,f}^{\cC}(L)$.

\noindent
$(ii)$ 
Write $\Sigma' \defeq P_{\bQ,f}\setminus\Sigma(\cC)$.
Then 
$P_{K,f}^{\cC,l} = \cs(K^{\Sigma'} \cap K(\mu_{l})/K) \setminus P_{K,l}
$ for $l \in \Sigma(\cC)$.

\end{rem}

\begin{prop}\label{2.8}
Let $l \in \Sigma(\cC)$, and $H \subset G_{K}^{\cC}$ be a closed subgroup of $l$-decomposition type.
Then there exists a unique element $\overline{\p} \in P_{K,f}^{\cC,l}(K^{\cC})$ 
such that $H \subset D_{\overline{\p}}(K^{\cC}/K)$.
\end{prop}
\begin{proof}
The uniqueness follows from Proposition \ref{1.4}.
Let us prove the existence.
If $l=2$, we may assume that $\mu_4 \subset K$ by Lemma \ref{2.5}.
Let $U$ be any open subgroup of $G_{K}^{\cC}$ containing $H$.
Write $M$, $L$ for the subextensions of $K^{\cC}/K$ corresponding to $H$, $U$, respectively.
By Lemma \ref{2.6} and taking the inductive limits over $U$, 
we obtain the injective homomorphism 
$H^2(H,\bF_l) \hookrightarrow \ilim_{U}\bigoplus_{\p \in P_L} H^2(D_{\p}(K^{\cC}/L),\bF_l)$.
Moreover, by \cite[Lemma 2.13]{Ivanov2}, the canonical homomorphism 
$\ilim_{U}\bigoplus_{\p \in P_L} H^2(D_{\p}(K^{\cC}/L),\bF_l) \to \prod_{\p \in P_M}H^2(D_{\p}(K^{\cC}/M),\bF_l)$
is injective.
Since $H^2(H,\bF_l) \simeq \bF_l$ and $H^2(D_{\p}(K^{\cC}/M),\bF_l) = 0$ for $\p \in P_{M,\infty}$, 
there exists $\q \in P_{M,f}$ such that $H^2(D_{\q}(K^{\cC}/M),\bF_l) \not = 0$.
Then $D_{\q}(K^{\cC}/M)$ is not isomorphic to $\bZ_l$, and therefore it is open in $H$ (see \cite[Lemma 2.2 (i), (ii)]{Ivanov}).
Thus, by Lemma \ref{2.5}, 
there exists $\overline{\p} \in P_{K^{\cC},f}$ 
such that $H \subset D_{\overline{\p}}(K^{\cC}/K)$.
Then  $\overline{\p}$ does not lie above $l$ by Lemma \ref{2.3}, and we have $\mu_{l} \subset (K_{\overline{\p}|_{K}})^{\cC}$ by Lemma \ref{2.2}. Hence $\overline{\p} \in P_{K,f}^{\cC,l}(K^{\cC})$.
\end{proof}

\begin{theorem}\label{2.9}
Let $l \in \Sigma(\cC)$.
Then the set $\Dec_{K,f}^{\cC,l}$ can be recovered group-theoretically from $G_{K}^{\cC}$.
\end{theorem}
\begin{proof}
Let $U$ be any open 
normal 
subgroup of $G_{K}^{\cC}$.
Write $L$ for the finite Galois subextension of $K^{\cC}/K$ corresponding to $U$.
Set 
$$\Syl_l(U) \defeq \{ H \subset U \mid \text{$H$ is an $l$-Sylow subgroup of $D_{\overline{\p}}(K^{\cC}/L)$ for some $\overline{\p} \in P_{K,f}^{\cC,l}(K^{\cC})$} \}.$$
Note that $G_{K}^{\cC}$ acts on $\Syl_l(U)$ by conjugation.
By Lemma \ref{2.2} and Proposition \ref{2.8}, $\Syl_l(U)$ is the set of maximal elements of the set 
$$\{ H \subset U \mid \text{$H$ is a closed subgroup of $U$ of $l$-decomposition type} \},$$
and hence can be recovered group-theoretically from $U$.
We have a $G_{K}^{\cC}$-equivariant surjection 
$$\psi_U : \Syl_l(U) \twoheadrightarrow P_{L,f}^{\cC,l}(=P_{K,f}^{\cC,l}(L)),\ D_{\overline{\p},l} 
\mapsto \overline{\p}|_{L},$$ 
where $D_{\overline{\p},l}$ 
is an $l$-Sylow subgroup of $D_{\overline{\p}}(K^{\cC}/L)$ for $\overline{\p} \in P_{K,f}^{\cC,l}(K^{\cC})$.
It follows from Proposition \ref{2.8} that for $H, H' \in \Syl_l(U)$, 
$\psi_U(H) = \psi_U(H')$ if and only if $H$ and $H'$ are conjugate in $U$.
Therefore, we can define a group-theoretical equivalence relation $\sim$ on $\Syl_l(U)$ 
such that $\psi_U$ induces a $G_{K}^{\cC}$-equivariant bijection 
$\overline{\psi_U} : \Syl_l(U)/\sim \isom P_{L,f}^{\cC,l}$.
Let $U'$ be any 
open normal subgroup of $G_{K}^{\cC}$ with $U' \subset U$, corresponding to a subextension $L'$ of $K^{\cC}/K$.
We can construct (possibly non-canonically) a 
group-theoretical 
map 
$\alpha_{U' \subset U}: \Syl_l(U') \to \Syl_l(U)$ 
which sends $H' \in \Syl_l(U')$ to some $H \in \Syl_l(U)$ such that $H' \subset H$. 
Then $\alpha_{U' \subset U}$ induces a 
$G_{K}^{\cC}$-equivariant 
map 
$\overline{\alpha_{U' \subset U}}: \Syl_l(U')/\sim \to \Syl_l(U)/\sim$, which is independent of the choice of $\alpha_{U' \subset U}$.
Further, the following diagram of $G_{K}^{\cC}$-equivariant maps commutes:
$$
\xymatrix{
\Syl_l(U')/\sim\ar[r]^-{\sim}_-{\overline{\psi_{U'}}} \ar[d]_-{\overline{\alpha_{U' \subset U}}} &P_{L',f}^{\cC,l}\ar[d]\\
\Syl_l(U)/\sim\ar[r]^-{\sim}_-{\overline{\psi_U}}&P_{L,f}^{\cC,l},
}
$$
where the vertical map on the right is the restriction of primes.
Therefore, by taking the projective limits over $U$, we obtain a $G_{K}^{\cC}$-equivariant bijection 
$$\plim_{U} \Syl_l(U)/\sim \longisom \plim_{L} P_{L,f}^{\cC,l} \simeq P_{K,f}^{\cC,l}(K^\cC).$$
Thus, the set $\Dec_{K,f}^{\cC,l}$ can be characterized as the set of stabilizers in $G_{K}^{\cC}$ of elements in the $G_{K}^{\cC}$-set $\plim_{U} \Syl_l(U)/\sim$, which is constructed group-theoretically from $G_{K}^{\cC}$.
\end{proof}

By Lemma \ref{2.0} and Theorem \ref{2.9}, we obtain the following.

\begin{cor}\label{2.10}
The set $\Dec_{K,f}^{\cC}$ can be recovered group-theoretically from $G_{K}^{\cC}$.
\end{cor}

\begin{lem}\label{2.11}
Let $l \in \Sigma(\cC)$, $p$ be a prime number, and $\kappa$ a $p$-adic field. Then $G_{\kappa}^{\cC,(l),\ab,/\tor}$ is a free $\bZ_l$-module of rank $[\kappa:\bQ_p] + 1$ (resp. $1$) if $p=l$ (resp. $p\neq l$).
\end{lem}
\begin{proof}
The assertion follows from local class field theory.
\end{proof}

In the rest of this section, we recover 
the pro-$\cC$ factor 
of the 
$\Sigma(\cC)$-adic 
cyclotomic character of $G_{K}^{\cC}$.
\begin{lem}\label{cyc1}
Let $l \in \Sigma(\cC)$, 
and $\kappa$ be a $p$-adic field for some prime number $p \neq l$.
Assume 
$\mu_{l} \subset \kappa^{\cC}$.
Then 
the $l$-adic cyclotomic character $\chi_{\kappa,l} \colon G_{\kappa}\rightarrow {\bZ_l}^{\ast}$ factors as 
$$G_{\kappa} \twoheadrightarrow G_{\kappa}^{\cC} \overset{\chi_{\kappa,l}^{\cC}}\longrightarrow {\bZ_l}^{\ast, \cC} \hookrightarrow {\bZ_l}^{\ast},$$ where we consider ${\bZ_l}^{\ast, \cC}$ as a direct sum factor of ${\bZ_l}^{\ast}$, 
and $\chi_{\kappa,l}^{\cC}$ can be recovered group-theoretically from $G_{\kappa}^{\cC}$ (and $l$).
\end{lem}
\begin{proof}
By Lemma \ref{2.2}, 
the maximal element of 
$$\{ U \subset G_{\kappa}^{\cC} 
\mid \text{$U$ is an open subgroup of $G_{\kappa}^{\cC}$ and $U^{(l)}$ is of $l$-decomposition type} \}$$
corresponds to $\kappa(\mu_{l})$.
Write $I (= G_{\kappa(\mu_{l^{\infty}})}^{(l)} \simeq \bZ_l)$ for the inertia subgroup of $G_{\kappa(\mu_{l})}^{(l)}$.
Then $I$ 
is the unique maximal closed normal pro-cyclic subgroup of $G_{\kappa(\mu_{l})}^{(l)}$ (see the proof of Lemma \ref{2.2} and \cite[Lemma 2.2 (iii)]{Ivanov}).
By \cite[(7.5.2) Proposition]{NSW}, 
the homomorphism 
$$
G(\kappa(\mu_{l^{\infty}})/\kappa) \to \Aut(I) \isom {\bZ_l}^{\ast}$$
induced by 
the conjugation action of $G(\kappa(\mu_{l})^{(l)}/\kappa)$ on $I$ coincides with the $l$-adic cyclotomic character, 
where $\Aut(I) \isom {\bZ_l}^{\ast}$ is 
the inverse of ${\bZ_l}^{\ast} \isom \Aut(I)$, $a\mapsto (g\mapsto g^a)$.
Further, $\Sigma(\cC)$ can be recovered group-theoretically from $G_{\kappa}^{\cC}$ by a similar argument to the proof of Lemma \ref{2.0}.
The assertion follows from these, together with the fact that ${\bZ_l}^{\ast, \cC} = {\bZ_l}^{\ast, \Sigma(\cC)}$.
\end{proof}

\begin{lem}\label{cyc2}
Let $l \in \Sigma(\cC)$.
Then $G_{K}^{\cC}$ is topologically generated by 
the subgroups belonging to 
$\Dec_{K,f}^{\cC,l}$.
\end{lem}
\begin{proof}
Write 
$\Sigma' \defeq P_{\bQ,f}\setminus\Sigma(\cC)$, and 
$H$ for the closed normal subgroup of $G_{K}^{\cC}$ topologically generated by 
the subgroups belonging to 
$\Dec_{K,f}^{\cC,l}$.
By 
Remark \ref{csset} $(ii)$ and 
\cite[Chapter VII, (13.9) Proposition]{Neukirch3}, 
the subextension of $K^{\cC}/K$ corresponding to $G_{K}^{\cC}/H$ is contained in $K^{\Sigma'} \cap K(\mu_{l})$, 
and hence coincides with $K$ since $K^{\cC} \cap K^{\Sigma'} \cap K(\mu_{l}) =K$.
Thus, $H = G_{K}^{\cC}$.
\end{proof}

\begin{prop}\label{cyc3}
Let $l \in \Sigma(\cC)$.
Then $\chi_{K,l}^{\cC}$ can be recovered group-theoretically from $G_{K}^{\cC}$.
\end{prop}
\begin{proof}
By Theorem \ref{2.9}, the decomposition groups of $P_{K,f}^{\cC,l}(K^{\cC})$ can be recovered group-theoretically from $G_{K}^{\cC}$.
By Proposition \ref{1.2} and Lemma \ref{cyc1}, 
for $\overline{\p} \in P_{K,f}^{\cC,l}(K^{\cC})$, 
$\chi_{K,l}^{\cC}|_{D_{\overline{\p}}(K^{\cC}/K)}$ can be recovered group-theoretically from $D_{\overline{\p}}(K^{\cC}/K)$.
Thus, by Lemma \ref{cyc2}, 
gluing together these local characters on the decomposition groups of $P_{K,f}^{\cC,l}(K^{\cC})$ yields $\chi_{K,l}^{\cC}$.

\end{proof}

\section{The local correspondence}
In this section, 
by using the results in \S 2, 
we obtain the ``local correspondence'': 
a one-to-one correspondence between the sets of decomposition groups 
induced by an isomorphism of Galois groups, 
and study its properties.

\begin{definition}\label{3.2}
For $i=1,2$, let $K_i$ be a number field, 
$M_i/K_i$ a Galois extension, $S_i \subset P_{K_i,f}$ a set of nonarchimedean primes of $K_i$, 
and $\sigma :G(M_1/K_1)\isom G(M_2/K_2)$ an isomorphism.
We say that 
a map $\phi \colon S_1(M_1) \to S_2(M_2)$ is a weak local correspondence between $S_1$ and $S_2$ for $\sigma$ 
if 
for any $\overline{\p}_1 \in S_1(M_1)$, 
$\sigma(D_{\overline{\p}_1}(M_1/K_1))$ coincides with a decomposition group of $G(M_2/K_2)$ 
and $\sigma(D_{\overline{\p}_1}(M_1/K_1)) = D_{\phi(\overline{\p}_1)}(M_2/K_2)$.
In other words, 
$\phi \colon S_1(M_1) \to S_2(M_2)$ is a weak local correspondence between $S_1$ and $S_2$ for $\sigma$ 
if 
$\sigma$ induces an injective map 
$$\Dec(M_1/K_1,S_1(M_1)) \hookrightarrow \Dec(M_2/K_2,S_2(M_2)) ; D \mapsto \sigma(D)$$ 
such that 
the following diagram commutes:
\begin{equation}\label{a}
\xymatrix{
S_1(M_1) \ar@{->>}[r]^-{}_-{} \ar[d]_-{\phi} &\Dec(M_1/K_1,S_1(M_1))\ar@{^{(}-{>}}[d]\\
S_2(M_2)\ar@{->>}[r]^-{}_-{}&\Dec(M_2/K_2,S_2(M_2)),
}
\end{equation}
where the horizontal arrows are the canonical surjections mapping primes to their decomposition groups.

Let $\phi \colon S_1(M_1) \to S_2(M_2)$ be a weak local correspondence between $S_1$ and $S_2$ for $\sigma$.
We say that $\phi$ is Galois equivariant (for $\sigma$) if
for each $g \in G(M_1/K_1)$ and $\bap \in S_1(M_1)$, 
$\phi(g \bap) = \sigma(g)\phi(\bap)$.
We say that $\phi$ satisfies 
condition $(\Char)$ (resp. condition $(\Deg)$)
if for any $\overline{\p}_1 \in S_1(M_1)$, 
the residue characteristics (resp. the local degrees) of $\overline{\p}_1|_{K_1}$ and $\phi(\overline{\p}_1)|_{K_2}$ coincide.
We say that $\phi$ satisfies condition $(\Frob)$
if 
for any $\overline{\p}_1 \in S_1(M_1)$, 
$\sigma|_{D_{\overline{\p}_1}} \colon D_{\overline{\p}_1} \isom D_{\phi(\overline{\p}_1)}$ induces 
an isomorphism $D_{\overline{\p}_1}/I_{\overline{\p}_1} \isom D_{\phi(\overline{\p}_1)}/I_{\phi(\overline{\p}_1)}$,
and the Frobenius elements correspond to each other under this isomorphism.
We say that $\phi$ is a local correspondence between $S_1$ and $S_2$ for $\sigma$ 
if $\phi$ is bijective.
\end{definition}

\begin{lem}\label{3.3}
We use the notations in Definition \ref{3.2}.
Let $U_1$ be an open subgroup of $G(M_1/K_1)$. 
Set $U_2 \defeq \sigma(U_1)$.
For $i=1,2$, write $L_i$ for the finite subextension of $M_i/K_i$ corresponding to $U_i$.
Then the following hold.
\begin{itemize}
\item[$(i)$]
Let $\phi \colon S_1(M_1) \to S_2(M_2)$ be a weak local correspondence (resp. a local correspondence) between $S_1$ and $S_2$ for $\sigma$.
Then $\phi$ is a weak local correspondence (resp. a local correspondence) between $S_1(L_1)$ and $S_2(L_2)$ for $\sigma|_{U_1}\colon U_1 \isom U_2$.
\item[$(ii)$]
Let $\phi \colon S_1(M_1) \to S_2(M_2)$ be a weak local correspondence between $S_1$ and $S_2$ for $\sigma$ satisfying one of the conditions $(\Char)$, $(\Deg)$, or $(\Frob)$.
Then $\phi$ is a weak local correspondence between $S_1(L_1)$ and $S_2(L_2)$ for $\sigma|_{U_1}\colon U_1 \isom U_2$ satisfying the same condition.
\end{itemize}
\end{lem}
\begin{proof}
$(i)$ follows immediately from the following commutative diagram: 
\begin{equation*}\label{b}
\xymatrix{
S_1(M_1) \ar@{->>}[r]^-{}_-{} \ar[d]_-{\phi} &\Dec(M_1/K_1,S_1(M_1))\ar@{^{(}-{>}}[d]\ar@{->>}[r]^-{}_-{} &\Dec(M_1/L_1,S_1(M_1))\ar@{^{(}-{>}}[d]\\
S_2(M_2)\ar@{->>}[r]^-{}_-{}&\Dec(M_2/K_2,S_2(M_2))\ar@{->>}[r]^-{}_-{} &\Dec(M_2/L_2,S_2(M_2)),
}
\end{equation*}
where the horizontal arrows on the right are the canonical surjections induced by the restriction, 
and the vertical map on the right is induced by $\sigma|_{U_1}$.

$(ii)$ follows from the fact that for $i=1,2$ and $\overline{\p} \in P_{M_i,f}$, 
\begin{equation*}
  \begin{gathered}
(D_{\overline{\p}}(M_i/K_i) \colon D_{\overline{\p}}(M_i/L_i)) = [L_{i,\overline{\p}|_{L_i}} \colon K_{i,\overline{\p}|_{K_i}}], \\
I_{\overline{\p}}(M_i/K_i) \cap G(M_i/L_i) = I_{\overline{\p}}(M_i/L_i), \\
(D_{\overline{\p}}(M_i/K_i)/I_{\overline{\p}}(M_i/K_i) \colon D_{\overline{\p}}(M_i/L_i)/I_{\overline{\p}}(M_i/L_i)) = f_{\overline{\p}|_{L_i},L_i/K_i},
  \end{gathered}
\end{equation*}
and 
the Frobenius element in $D_{\overline{\p}}(M_i/L_i)/I_{\overline{\p}}(M_i/L_i)$ maps to the $f_{\overline{\p}|_{L_i},L_i/K_i}$-th power of the Frobenius element in $D_{\overline{\p}}(M_i/K_i)/I_{\overline{\p}}(M_i/K_i)$ under 
\begin{equation*}
  D_{\overline{\p}}(M_i/L_i)/I_{\overline{\p}}(M_i/L_i) \hookrightarrow D_{\overline{\p}}(M_i/K_i)/I_{\overline{\p}}(M_i/K_i).\tag*{\qedhere}
\end{equation*}
\end{proof}

\begin{lem}\label{3.4}
We use the notations in Definition \ref{3.2}.
Assume that the canonical surjection $S_2(M_2) \twoheadrightarrow \Dec(M_2/K_2,S_2(M_2))$ is bijective.
Then the following hold.
\begin{itemize}

\item[$(i)$] 
There exists at most one weak local correspondence $S_1(M_1) \to S_2(M_2)$ between $S_1$ and $S_2$ for $\sigma$.

\item[$(ii)$] 
If the weak local correspondence between $S_1$ and $S_2$ for $\sigma$ exists, then it is Galois equivariant.

\end{itemize}
Further, 
assume that the weak local correspondence between $S_1$ and $S_2$ for $\sigma$ exists, and denote it by $\phi \colon S_1(M_1) \to S_2(M_2)$.
Let $U_1$ be a closed normal subgroup of $G(M_1/K_1)$. 
Set $U_2 \defeq \sigma(U_1)$.
For $i=1,2$, write $L_i$ for the Galois subextension of $M_i/K_i$ corresponding to $U_i$.
Then the following hold.
\begin{itemize}
\item[$(iii)$] $\phi$ induces a 
Galois equivariant 
weak local correspondence
$\overline{\phi} \colon 
S_1(L_1) \to S_2(L_2)$ between $S_1$ and $S_2$ for the isomorphism $G(L_1/K_1)\isom G(L_2/K_2)$ induced by $\sigma$.
If $\phi$ is a local correspondence, then $\overline{\phi}$ is also a local correspondence.
If $\phi$ satisfies one of the conditions $(\Char)$, $(\Deg)$, or $(\Frob)$, then $\overline{\phi}$ also satisfies the same condition.

\end{itemize}
\end{lem}
\begin{proof}
$(i)$ 
and $(ii)$ 
follow immediately from the commutative diagram (\ref{a}), in which the maps, possibly except $\phi$, 
are Galois equivariant.
Let us prove $(iii)$.
By $(ii)$, $\phi$ induces a map $\overline{\phi} \colon S_1(L_1)(\simeq S_1(M_1)/U_1) \to S_2(L_2)(\simeq S_2(M_2)/U_2)$, 
which is compatible with the map $\Dec(L_1/K_1,S_1(L_1)) \hookrightarrow \Dec(L_2/K_2,S_2(L_2))$ induced by the isomorphism $G(L_1/K_1)\isom G(L_2/K_2)$ induced by $\sigma$.
Thus, $\overline{\phi}$ is also a weak local correspondence.
If $\phi$ is bijective, $\overline{\phi}$ is also bijective.
The last assertion follows from the fact that for $i=1,2$ and $\overline{\p} \in P_{M_i,f}$, 
$I_{\overline{\p}|_{L_i}}(L_i/K_i) =\Image(I_{\overline{\p}}(M_i/K_i) \hookrightarrow D_{\overline{\p}}(M_i/K_i) \twoheadrightarrow D_{\overline{\p}|_{L_i}}(L_i/K_i))$
and 
the Frobenius element in $D_{\overline{\p}}(M_i/K_i)/I_{\overline{\p}}(M_i/K_i)$ maps to the Frobenius element in $D_{\overline{\p}|_{L_i}}(L_i/K_i)/I_{\overline{\p}|_{L_i}}(L_i/K_i)$.
\end{proof}

\begin{theorem}\label{3.5}
For $i=1,2$, let $K_i$ be a number field, $\cC_i$ a nontrivial full class of finite groups and $\sigma :G_{K_1}^{\cC_1}\isom G_{K_2}^{\cC_2}$ an isomorphism.
Then there exists a unique local correspondence $\phi
$ between $P_{K_1,f}^{\cC_1}$ and $P_{K_2,f}^{\cC_2}$ for $\sigma$.
Write $\Sigma \defeq \Sigma(\cC_1)$.
Let $U_1$ be an open subgroup of $G_{K_1}^{\cC_1}$. 
Set $U_2 \defeq \sigma(U_1)$.
For $i=1,2$, write $L_i$ for the subextension of $K_i^{\cC_i}/K_i$ corresponding to $U_i$.
Then the following hold.
\begin{itemize}
\item[$(i)$] $\phi$ is 
the unique local correspondence between $P_{L_1,f}^{\cC_1}$ and $P_{L_2,f}^{\cC_2}$ for $\sigma|_{U_1}\colon U_1 \isom U_2$.

\item[$(ii)$] $\phi|_{(P_{K_1,f}^{\cC_1} \cap \Sigma(K_1))(K_1^{\cC_1})}$ is 
the unique local correspondence between $P_{L_1,f}^{\cC_1} \cap \Sigma(L_1)$ 
and $P_{L_2,f}^{\cC_2} \cap \Sigma(L_2)$ for $\sigma|_{U_1}\colon U_1 \isom U_2$, 
satisfying condition $(\Char)$ and condition $(\Deg)$.

\end{itemize}
Further, let $V_1$ be a closed normal subgroup of $U_1$. 
Set $V_2 \defeq \sigma(V_1)$.
For $i=1,2$, write $M_i$ for the Galois subextension of $K_i^{\cC_i}/L_i$ corresponding to $V_i$.
Write $\overline{\sigma|_{U_1}} \colon G(M_1/L_1)\isom G(M_2/L_2)$ for the isomorphism induced by $\sigma|_{U_1}$.
Then the following hold.
\begin{itemize}
\item[$(i)'$] $\phi$ 
induces a local correspondence between $P_{L_1,f}^{\cC_1}$ and $P_{L_2,f}^{\cC_2}$ for $\overline{\sigma|_{U_1}}$.

\item[$(ii)'$] $\phi|_{(P_{K_1,f}^{\cC_1} \cap \Sigma(K_1))(K_1^{\cC_1})}$
induces a local correspondence between $P_{L_1,f}^{\cC_1} \cap \Sigma(L_1)$ and $P_{L_2,f}^{\cC_2} \cap \Sigma(L_2)$ for $\overline{\sigma|_{U_1}}$, 
satisfying condition $(\Char)$ and condition $(\Deg)$.

\end{itemize}
Moreover, these local correspondences are all Galois equivariant.
\end{theorem}
\begin{proof}
We have that $\Sigma \defeq \Sigma(\cC_1) = \Sigma(\cC_2)$ by Lemma \ref{2.0}.
$(i)'$, $(ii)'$
follow from $(i)$, $(ii)$, 
Proposition \ref{1.4} and Lemma \ref{3.4} $(iii)$.
Let us prove $(i)$, $(ii)$.
The uniqueness follows from 
Proposition \ref{1.4} and 
Lemma \ref{3.4} $(i)$.
By Remark \ref{csset} $(i)$ and Lemma \ref{3.3}, 
it suffices to prove $(i)$ and $(ii)$ in the case $U_1=G_{K_1}^{\cC_1}$.
By Corollary \ref{2.10}, $\sigma$ induces a bijection $\Dec_{K_1,f}^{\cC_1} \isom \Dec_{K_2,f}^{\cC_2}$.
By Proposition \ref{1.4}, the canonical surjection $P_{K_i,f}^{\cC_i}(K_i^{\cC_i}) \twoheadrightarrow \Dec_{K_i,f}^{\cC_i}$ is bijective for $i=1,2$.
Thus, there exists a unique local correspondence $\phi \colon P_{K_1,f}^{\cC_1}(K_1^{\cC_1}) \isom P_{K_2,f}^{\cC_2}(K_2^{\cC_2})
$ between $P_{K_1,f}^{\cC_1}$ and $P_{K_2,f}^{\cC_2}$ for $\sigma$.
By Lemma \ref{2.11}, 
$\phi|_{(P_{K_1,f}^{\cC_1} \cap \Sigma(K_1))(K_1^{\cC_1})}$ is a local correspondence between $P_{K_1,f}^{\cC_1} \cap \Sigma(K_1)$ and $P_{K_2,f}^{\cC_2} \cap \Sigma(K_2)$ for $\sigma$, 
satisfying condition $(\Char)$ and condition $(\Deg)$.
The Galois equivariances follow from Lemma \ref{3.4} $(ii)$, $(iii)$.
\end{proof}

Let $l$ be a prime number.
We set 
$\Gamma_K = \Gamma_{K, l} \defeq G_{K}^{(l),\ab,/\tor}$.
Then $\Gamma_K$ is a free $\bZ_l$-module of finite rank.
Note that if $l \in \Sigma(\cC)$, $\Gamma_{K, l} = G_{K}^{\cC,(l),\ab,/\tor}$.
Set $r_l(K) \defeq \rank_{\bZ_l}\Gamma_K$, 
and write $K^{(\infty)}=K^{(\infty,l)}$ for the extension of $K$ corresponding to $\Gamma_{K}$.
By \cite[(10.3.20) Proposition]{NSW}, 
we have $r_\bC(K)+1 \leq r_l(K) \leq [K:\bQ]$, 
where $r_\bC(K)$ is the number of complex primes of $K$.
For $\p \in P_{K,f} \setminus P_{K,l}$, $\p$ is unramified in $K^{(\infty)}/K$.
Then 
we write $\Frob_{\p}$ for the Frobenius element in $\Gamma_{K}$ at $\p$.
For a finite extension $L/K$, 
define $\pi_{L/K} = \pi_{L/K, l}$ to be the canonical homomorphism $\Gamma_{L} \to \Gamma_{K}$.

We set $\ltilde \defeq l$ (resp. $\ltilde \defeq 4$) for $l\neq2$ (resp. $l=2$).
Write $\pr_1 \colon {\bZ_l}^{\ast}\twoheadrightarrow {1+\tilde{l}\bZ_l}$ for the first projection of the decomposition ${\bZ_l}^{\ast}=({1+\tilde{l}\bZ_l})\times{({\bZ_l}^{\ast})_{\tor}}$.
Then the composite of $\chi_{K,l} \colon G_{K}\rightarrow {\bZ_l}^{\ast}$ and $\pr_1$ 
factors as $G_{K}\twoheadrightarrow \Gamma_{K, l}\rightarrow 1+\ltilde\bZ_l$, 
where we write $
w_{K, l}:\Gamma_{K, l}\rightarrow 1+\ltilde\bZ_l$ for the second homomorphism.

\begin{prop}\label{3.6}
Let $l \in \Sigma(\cC)$.
Then $w_{K, l}$ can be recovered group-theoretically from $G_{K}^{\cC}$.
\end{prop}
\begin{proof}\footnote{In \cite[Appendix A]{Saidi-Tamagawa2}, 
$w_{K, l}$ is recovered group-theoretically from the maximal metabelian quotient of $G_{K}^{(l)} = G_{K}^{\cC,(l)}$.}
The composite of $\chi_{K,l}^{\cC}$ and 
$\pr_1 \colon {\bZ_l}^{\ast}\twoheadrightarrow {1+\tilde{l}\bZ_l}$
factors as 
$$G_{K}^{\cC} \twoheadrightarrow G_{K}^{\cC,(l),\ab,/\tor}=\Gamma_{K, l} \overset{w_{K, l}}\rightarrow 1+\ltilde\bZ_l.$$
The assertion follows from 
this, together with 
Proposition \ref{cyc3}.
\end{proof}

\begin{lem}\label{Frob}
We use the notations in Theorem \ref{3.5}.
Let $l \in \Sigma$, $S_i \subset P_{L_i,f} \setminus (P_{L_i,l} \cup \Ram(L_i/\bQ)(L_i))$ for $i=1,2$, and $\phi'$ a 
weak 
local correspondence between $S_1$ and $S_2$ for the isomorphism $\Gamma_{L_1, l} \isom \Gamma_{L_2, l}$ induced by $\sigma|_{U_1}$.
Then $\phi'$ satisfies conditions $(\Char)$ and $(\Deg)$ if and only if $\phi'$ satisfies condition $(\Frob)$.
\end{lem}
\begin{proof}
For $i=1,2$ and $\p \in S_i$
above a prime number $p$, 
$I_{\p}(L_i^{(\infty,l)}/L_i)$ is trivial, $w_{L_i, l}|_{D_{\p}(L_i^{(\infty,l)}/L_i)}$ is injective, 
the composite 
$$\bZ_{>0} \setminus l\bZ 
\hookrightarrow {\bZ_l}^{\ast} \overset{\pr_1}\twoheadrightarrow {1+\tilde{l}\bZ_l}$$
is injective, 
and 
$$w_{L_i, l}(\Frob_{\p}) = \pr_1(p^{f_{\p,L_i/\bQ}}) = \pr_1(p^{[L_{i,\p}:\bQ_p]}).$$
The assertion follows from these facts and Proposition \ref{3.6}.
\end{proof}

\begin{lem}\label{U2}
Let $S_0\subset P_{\bQ,f}$, and $\phi \colon S_0(K_1) \isom S_0(K_2)$ be a bijection preserving the residue characteristics and the local degrees of all primes in $S_0(K_1)$. Then $S_0 \cap \cs(K_1/\bQ)=S_0 \cap \cs(K_2/\bQ)$. If $S_0 \not= \emptyset$, then $[K_1:\bQ]=[K_2:\bQ]$.
\end{lem}
\begin{proof}
By the assumption, $\phi$ induces a bijection $\phi|_{S_0(K_1) \cap P_{K_1,p}} \colon S_0(K_1) \cap P_{K_1,p} \isom S_0(K_2) \cap P_{K_2,p}$ for any $p \in P_{\bQ,f}$.
For $p \in S_0$, 
$p \in S_0 \cap \cs(K_i/\bQ)$ if and only if 
the local degrees of all primes in $S_0(K_i) \cap P_{K_i,p}$ are $1$, and hence $S_0 \cap \cs(K_1/\bQ)=S_0 \cap \cs(K_2/\bQ)$.
If $p \in S_0$, 
then the sum of the local degrees of all primes in $S_0(K_i) \cap P_{K_i,p}$ coincides with $[K_i : \bQ]$, 
and hence $[K_1 : \bQ] = [K_2 : \bQ]$.
\end{proof}

\begin{lem}\label{CharDeg}
We use the notations in Theorem \ref{3.5}.
Let $S_i \subset P_{K_i,f}$ for $i=1,2$, and $\phi'$ a 
local correspondence between $S_1$ and $S_2$ for $\sigma$, 
satisfying condition $(\Char)$.
Assume that $\underline{S_i} \cap \cs(K_i/\bQ) \not= \emptyset$ for $i=1,2$.
Then the following hold.
\begin{itemize}
\item[$(i)$] 
$\underline{S_1} \cap \cs(L_1/\bQ) = \underline{S_2} \cap \cs(L_2/\bQ)$ 
and $[L_1:\bQ]=[L_2:\bQ]$.

\item[$(ii)$] 
$\phi'|_{(\underline{S_1} \cap \cs(K_1/\bQ))(K_1^{\cC_1})}$ induces a local correspondence between $(\underline{S_1} \cap \cs(K_1/\bQ))(L_1)$ and $(\underline{S_2} \cap \cs(K_2/\bQ))(L_2)$ for $\overline{\sigma|_{U_1}}$, 
satisfying conditions $(\Char)$ and $(\Deg)$.

\item[$(iii)$] 
Let $l \in \Sigma$.
Then 
$\phi'|_{((\underline{S_1} \cap \cs(K_1/\bQ)) \setminus (\{ l \} \cup \Ram(L_1L_2/\bQ) ))(K_1^{\cC_1})}$ 
induces a local correspondence between 
$$((\underline{S_1} \cap \cs(K_1/\bQ)) \setminus (\{ l \} \cup \Ram(L_1L_2/\bQ) ))(L_1) \text{ and } ((\underline{S_2}\cap \cs(K_2/\bQ)) \setminus (\{ l \} \cup \Ram(L_1L_2/\bQ) ))(L_2)$$ 
for the isomorphism $\Gamma_{L_1, l} \isom \Gamma_{L_2, l}$ induced by $\sigma|_{U_1}$, satisfying conditions $(\Char)$, $(\Deg)$ and $(\Frob)$.

\end{itemize}

\end{lem}
\begin{proof}
By Proposition \ref{1.4} and Lemma \ref{3.4} $(iii)$, 
$\phi'$ induces a bijection $\overline{\phi'}:S_1 \isom S_2$, 
preserving the residue characteristics, 
that is, $\overline{\phi'}$ induces a bijection 
$$\overline{\phi'}|_{S_1 \cap P_{K_1,p}}:S_1 \cap P_{K_1,p} \isom S_2 \cap P_{K_2,p}$$ 
for any $p \in P_{\bQ,f}$.
For 
$p \in P_{\bQ,f}$, 
$p \in \underline{S_i} \cap \cs(K_i/\bQ)$ if and only if $\# (S_i \cap P_{K_i,p}) = [K_i : \bQ]$.
By the assumption, 
we have $[K_i : \bQ] = \max_{p \in P_{\bQ,f}} \# (S_i \cap P_{K_i,p})$, 
and hence $[K_1 : \bQ] = [K_2 : \bQ]$.
Therefore, 
$\underline{S_1} \cap \cs(K_1/\bQ) = \underline{S_2} \cap \cs(K_2/\bQ)$, 
and $\phi'|_{(\underline{S_1} \cap \cs(K_1/\bQ))(K_1^{\cC_1})}$ 
is 
a local correspondence between 
$(\underline{S_1} \cap \cs(K_1/\bQ))(K_1)$ and $(\underline{S_2} \cap \cs(K_2/\bQ))(K_2)$ for $\sigma$, 
satisfying condition $(\Char)$ and condition $(\Deg)$.
$(ii)$ follows from this, together with Lemma \ref{3.3} $(ii)$, Proposition \ref{1.4} and Lemma \ref{3.4} $(iii)$.
By $(ii)$ and Lemma \ref{3.4} $(iii)$, $\phi'|_{(\underline{S_1} \cap \cs(K_1/\bQ))(K_1^{\cC_1})}$ induces a bijection 
$$(\underline{S_1} \cap \cs(K_1/\bQ))(L_1) \isom (\underline{S_2} \cap \cs(K_2/\bQ))(L_2),$$ 
preserving the residue characteristics and the local degrees.
Therefore, $(i)$ follows from Lemma \ref{U2}.
By $(ii)$, 
$\phi'|_{((\underline{S_1} \cap \cs(K_1/\bQ)) \setminus (\{ l \} \cup \Ram(L_1L_2/\bQ) ))(K_1^{\cC_1})}$ 
induces a local correspondence between 
$$((\underline{S_1} \cap \cs(K_1/\bQ)) \setminus (\{ l \} \cup \Ram(L_1L_2/\bQ) ))(L_1) \text{ and } ((\underline{S_2}\cap \cs(K_2/\bQ)) \setminus (\{ l \} \cup \Ram(L_1L_2/\bQ) ))(L_2)$$ 
for the isomorphism $\Gamma_{L_1, l} \isom \Gamma_{L_2, l}$ induced by $\sigma|_{U_1}$, satisfying conditions $(\Char)$ and $(\Deg)$.
Further, by Lemma \ref{Frob}, 
this satisfies condition $(\Frob)$.
\end{proof}

\section{Recovering the residue characteristics
}

In this section, 
to prove that the local correspondence 
satisfies condition $(\Char)$ under some conditions, 
we study 
separatedness 
of the decomposition groups in $G(\bQ(\cup_{l \in \Sigma(\cC)}\mu_{l})/\bQ)^{\cC} = G(\bQ^{\cC} \cap \bQ(\cup_{l \in \Sigma(\cC)}\mu_{l})/\bQ)$.

Let $\Sigma \subset P_{\bQ,f}$ and $l_0 \in P_{\bQ,f}$.
For a prime number $l$, 
we will identify $G(\bQ(\mu_{l})/\bQ)$ with $\bF_l^{\ast} = (\bZ/l\bZ)^{\ast}$ via the mod $l$ cyclotomic character.
Write 
$$\Sigma_{l_0} \defeq (\Sigma \cap \cs(\bQ(\mu_{l_0})/\bQ)) \setminus \{ l_0 \} = \{ l \in \Sigma \mid l \equiv 1 \text{ mod } l_0 \},$$
and $K_{\Sigma,l_0}$ for the subfield of $K(\cup_{l \in \Sigma}\mu_{l})/K$ corresponding to $$G(K(\cup_{l \in \Sigma}\mu_{l})/K)/G(K(\cup_{l \in \Sigma}\mu_{l})/K)^{l_0}.$$
Then 
\begin{equation*}
\bF_l^{\ast}/\bF_l^{\ast,l_0} \simeq 
\begin{cases}
1, &l \in \Sigma \setminus \Sigma_{l_0}, \\
\bF_{l_0}, &l \in \Sigma_{l_0},
\end{cases}
\end{equation*}
and hence 
$$G(\bQ_{\Sigma,l_0}/\bQ) 
= \prod_{l \in \Sigma} G(\bQ(\mu_{l})/\bQ)/G(\bQ(\mu_{l})/\bQ)^{l_0}
= \prod_{l \in \Sigma} \bF_l^{\ast}/\bF_l^{\ast,l_0}
= \prod_{l \in \Sigma_{l_0}} \bF_l^{\ast}/\bF_l^{\ast,l_0} 
(\simeq \prod_{l \in \Sigma_{l_0}} \bF_{l_0}).$$

\begin{lem}\label{incomm1}
Let $l_0 \in P_{\bQ,f}$, and $p,q$ be distinct prime numbers in $P_{\bQ,f} \setminus \Sigma_{l_0}$.
Then 
$D_{p}(\bQ_{\Sigma,l_0}/\bQ) \subset D_{q}(\bQ_{\Sigma,l_0}/\bQ)$ 
if and only if 
there exists $i \in \{ 0,1,...,l_0 - 1 \}$ such that 
$$\Sigma\setminus \{ l_0 \} \subset \cs(\bQ(\mu_{l_0}, \sqrt[l_0]{pq^{-i}})/\bQ) \coprod (P_{\bQ,f} \setminus\cs(\bQ(\mu_{l_0})/\bQ)).$$
\end{lem}
\begin{proof}
By the theory of cyclotomic fields, 
$p,q$ are unramified in $\bQ_{\Sigma,l_0}/\bQ$, and 
the Frobenius elements in $G(\bQ_{\Sigma,l_0}/\bQ)= \prod_{l \in \Sigma_{l_0}} \bF_l^{\ast}/\bF_l^{\ast,l_0}$ at $p,q$
are $(p)_{l \in \Sigma_{l_0}}, (q)_{l \in \Sigma_{l_0}}$, respectively. 
Therefore, 
$D_{p}(\bQ_{\Sigma,l_0}/\bQ) \subset D_{q}(\bQ_{\Sigma,l_0}/\bQ)$ 
if and only if 
there exists $i \in \{ 0,1,...,l_0 - 1 \}$ such that 
for any $l \in \Sigma_{l_0}$, $pq^{-i} \in \bF_l^{\ast,l_0}$.
Let $l \in \Sigma_{l_0}$.
Write $L \defeq \bQ(\sqrt[l_0]{pq^{-i}})$ and $\mathcal O_{L}$ for its ring of integers.
Then, by \cite[Chapter III, \S 6, Corollary 2]{Serre2}, 
$\mathcal O_{L}[\frac{1}{pql_0}] = \bZ[\frac{1}{pql_0}, \sqrt[l_0]{pq^{-i}}]$.
Hence 
$$\mathcal O_{L} \otimes_{\bZ} \bF_l 
\simeq \mathcal O_{L}[\frac{1}{pql_0}] \otimes_{\bZ} \bF_l 
= \bZ[\frac{1}{pql_0}, \sqrt[l_0]{pq^{-i}}] \otimes_{\bZ} \bF_l 
\simeq \bF_l[t]/(t^{l_0} - pq^{-i}).$$
Since $\mu_{l_0} \subset \bF_l^{\ast}$, 
$pq^{-i} \in \bF_l^{\ast,l_0}$ if and only if $l \in \cs(L/\bQ)$.
Since $\bQ(\mu_{l_0}, \sqrt[l_0]{pq^{-i}})$ is the Galois closure of $L/\bQ$, we obtain $\cs(L/\bQ) = \cs(\bQ(\mu_{l_0}, \sqrt[l_0]{pq^{-i}})/\bQ)$.
By the definition of $\Sigma_{l_0}$, 
$\Sigma_{l_0} \subset \cs(\bQ(\mu_{l_0}, \sqrt[l_0]{pq^{-i}})/\bQ)$ 
if and only if 
$\Sigma\setminus \{ l_0 \} \subset \cs(\bQ(\mu_{l_0}, \sqrt[l_0]{pq^{-i}})/\bQ) \coprod (P_{\bQ,f} \setminus\cs(\bQ(\mu_{l_0})/\bQ))$.
\end{proof}

\begin{prop}\label{incomm2}
Let $p,q$ be distinct prime numbers in $P_{\bQ,f} \setminus \Sigma$.
Assume $D_{p}(\bQ_{\Sigma,l_0}/\bQ) \subset D_{q}(\bQ_{\Sigma,l_0}/\bQ)$ for any $l_0 \in \Sigma$.
Then $\delta_{\sup}(\Sigma) \leq \prod_{l_0 \in \Sigma} \frac{l_0-1}{l_0}$.
\end{prop}
\begin{proof}
Let $\Sigma' \subset \Sigma$ be any finite subset.
By Lemma \ref{incomm1}, 
there exists $i_{l_0} \in \{ 0,1,...,l_0 - 1 \}$ for $l_0 \in \Sigma'$ such that 
\begin{equation}\label{eq:incomm}
\Sigma\setminus \Sigma' \subset 
A_{\Sigma'} \defeq 
\cap_{l_0 \in \Sigma'}(
\cs(\bQ(\mu_{l_0}, \sqrt[l_0]{pq^{-i_{l_0}}})/\bQ) \coprod (P_{\bQ,f} \setminus\cs(\bQ(\mu_{l_0})/\bQ))).
\end{equation}
Let us calculate the Dirichlet density of $A_{\Sigma'}$.
Write $G \defeq G(\bQ(\cup_{l_0 \in \Sigma'} \{ \mu_{l_0}, \sqrt[l_0]{pq^{-i_{l_0}}} \})/\bQ)$, 
and $G_{l_0} \defeq G(\bQ(\mu_{l_0}, \sqrt[l_0]{pq^{-i_{l_0}}})/\bQ)$ for $l_0 \in \Sigma'$.
Then $G = \prod_{l_0 \in \Sigma'} G_{l_0}$, and 
$$G_{l_0} \simeq G(\bQ(\mu_{l_0}, \sqrt[l_0]{pq^{-i_{l_0}}})/\bQ(\mu_{l_0})) \rtimes G(\bQ(\mu_{l_0})/\bQ)
\simeq \bF_{l_0} \rtimes \bF_{l_0}^{\ast}$$ 
by Kummer theory.
By calculating the differents (cf. \cite[Chapter III, \S 6, Corollary 2]{Serre2}) or \cite[Chapter V, (3.3) Lemma]{Neukirch3}, $\bQ(\cup_{l_0 \in \Sigma'} \{ \mu_{l_0}, \sqrt[l_0]{pq^{-i_{l_0}}} \})/\bQ$ is unramified outside $\Sigma' \cup \{ p,q,\infty \}$.
Write $\Frob_l$ for the Frobenius element in $G$ at $l \in P_{\bQ,f} \setminus (\Sigma' \cup \{ p,q \})$, and 
$$X_{l_0} \defeq \{ 1_{G_{l_0}} \} \coprod \bF_{l_0} \rtimes (\bF_{l_0}^{\ast}\setminus\{ 1 \}) \subset G_{l_0}.$$
Then 
$$A_{\Sigma'} \setminus (\Sigma' \cup \{ p,q \}) 
= \{ l \in P_{\bQ,f} \setminus (\Sigma' \cup \{ p,q \}) \mid \Frob_l \in \prod_{l_0 \in \Sigma'} X_{l_0} \}.$$
Therefore, by the Chebotarev density theorem, 
\begin{equation*}
\begin{split} 
\delta(A_{\Sigma'}) 
&= \frac{ \# \prod_{l_0 \in \Sigma'} X_{l_0} }{ \# G}
=\frac{ \prod_{l_0 \in \Sigma'} (1+ l_0(l_0-2)) }{ \prod_{l_0 \in \Sigma'} l_0(l_0-1) }
=\prod_{l_0 \in \Sigma'} \frac{l_0-1}{l_0}.
\end{split}
\end{equation*}
By (\ref{eq:incomm}), we have $\delta_{\sup}(\Sigma) \leq \prod_{l_0 \in \Sigma'} \frac{l_0-1}{l_0}$.
Thus, the inequality in the assertion is obtained by passing to the limit over 
$\Sigma' \subset \Sigma$.
\end{proof}

\begin{lem}\label{incomm3}
$(i)$
If $\delta_{\sup}(\Sigma)>0$, then $\sum_{p \in \Sigma} p^{-1} = \infty$.

\noindent
$(ii)$
$\sum_{p \in \Sigma} p^{-1} = \infty$ if and only if $\prod_{p \in \Sigma} \frac{p-1}{p} = 0$.
\end{lem}
\begin{proof}
$(i)$: 
Take a sequence $\{ s_n \}_{n \in \bZ_{>0}}$ of real numbers with $s_n > 1$ such that 
$$\lim_{n \to \infty} s_n = 1 \text{ and } 
\lim_{n \to \infty} 
\frac{\sum_{p \in \Sigma} p^{-s_n}}{\log{\frac{1}{s_n-1}}} 
= \delta_{\sup}(\Sigma).$$
Then 
$$\sum_{p \in \Sigma} p^{-1} \geq \lim_{n \to \infty} \sum_{p \in \Sigma} p^{-s_n} = \delta_{\sup}(\Sigma) \lim_{n \to \infty}\log{\frac{1}{s_n-1}} = \infty.$$
$(ii)$: 
$\prod_{p \in \Sigma} \frac{p-1}{p} = 0$ if and only if $\prod_{p \in \Sigma} \frac{p}{p-1} = \infty$.
Since 
$$\log(\prod_{p \in \Sigma} \frac{p}{p-1}) = \sum_{p \in \Sigma}\sum_{n \in \bZ_{>0}}\frac{1}{np^n} 
\text{ and } 
\sum_{p \in \Sigma}\sum_{n \in \bZ_{>1}}\frac{1}{np^n}$$ 
converges, 
$\prod_{p \in \Sigma} \frac{p}{p-1} = \infty$ if and only if $\sum_{p \in \Sigma} \frac{1}{p} = \infty$.
\end{proof}

\begin{prop}\label{incomm4}
Let $p,q$ be distinct prime numbers.
Write $D_{p}, D_{q}$ for the decomposition groups of $G(\bQ(\cup_{l \in \Sigma}\mu_{l})/\bQ)^{\Sigma} = G(\bQ^{\Sigma} \cap \bQ(\cup_{l \in \Sigma}\mu_{l})/\bQ)$ at $p,q$, respectively. 
Assume $\delta_{\sup}(\Sigma) > 0$.
Then 
$D_{p} \cap D_{q}$ 
is not open in either $D_{p}$ or $D_{q}$.
\end{prop}
\begin{proof}
By symmetry, we are reduced to proving that $D_{p} \cap D_{q}$ is not open in $D_{p}$.
Assume $D_{p} \cap D_{q}$ is open in $D_{p}$.
Replacing $\Sigma$ by 
$$\Sigma \setminus (\{ l \in P_{\bQ,f} \mid l \text{ divides } (D_{p} \colon D_{p} \cap D_{q}) \} \cup \{ p,q \}),$$
we may assume that $D_{p} \subset D_{q}$ 
and that $\Sigma \cap \{ p,q \} = \emptyset$.
Then $D_{p}(\bQ_{\Sigma,l_0}/\bQ) \subset D_{q}(\bQ_{\Sigma,l_0}/\bQ)$ for any $l_0 \in \Sigma$.
By Proposition \ref{incomm2} and Lemma \ref{incomm3}, we have 
$$0 < \delta_{\sup}(\Sigma) \leq \prod_{l_0 \in \Sigma} \frac{l_0-1}{l_0} = 0,$$
a contradiction.
\end{proof}

\begin{prop}\label{incomm5}
We use the notations in Theorem \ref{3.5}.
Assume $\delta_{\sup}(\Sigma) > 0$.
Then the following hold.
\begin{itemize}
\item[$(i)$] The local correspondence in Theorem \ref{3.5} $(i)'$ 
satisfies condition $(\Char)$.

\item[$(ii)$] 
$\underline{P_{K_1,f}^{\cC_1}} \cap \cs(L_1/\bQ) = \underline{P_{K_2,f}^{\cC_2}} \cap \cs(L_2/\bQ)$.

\item[$(iii)$] $\phi|_{(\underline{P_{K_1,f}^{\cC_1}} \cap \cs(K_1/\bQ))(K_1^{\cC_1})}$
induces a local correspondence between 
$(\underline{P_{K_1,f}^{\cC_1}} \cap \cs(K_1/\bQ))(L_1)$ and $(\underline{P_{K_2,f}^{\cC_2}} \cap \cs(K_2/\bQ))(L_2)$ for $\overline{\sigma|_{U_1}}$, 
satisfying conditions $(\Char)$ and $(\Deg)$.

\end{itemize}
\end{prop}
\begin{proof}
By Lemma \ref{3.3} $(ii)$, Proposition \ref{1.4} and Lemma \ref{3.4} $(iii)$, 
for $(i)$, it suffices to show that 
the local correspondence $\phi$ between $P_{K_1,f}^{\cC_1}$ and $P_{K_2,f}^{\cC_2}$ for $\sigma$ 
satisfies condition $(\Char)$.
For $i=1,2$, by Lemma \ref{2.0} and Proposition \ref{cyc3}, 
the canonical homomorphism 
$$G_{K_i}^{\cC_i} \to 
G(\bQ^{\Sigma} \cap \bQ(\cup_{l \in \Sigma}\mu_{l^\infty})/\bQ) (=G(\bQ(\cup_{l \in \Sigma}\mu_{l^\infty})/\bQ)^{\Sigma} =\prod_{l \in \Sigma} {\bZ_l}^{\ast, \Sigma}), \  \tau \mapsto \tau|_{\bQ^{\Sigma} \cap \bQ(\cup_{l \in \Sigma}\mu_{l^\infty})}$$
can be recovered group-theoretically from $G_{K_i}^{\cC_i}$.
For $\overline{\p} \in P_{K_i,f}^{\cC_i}(K_i^{\cC_i})$ 
and $p \in P_{\bQ,f}$, by Proposition \ref{incomm4}, 
the image of $D_{\overline{\p}}(K_i^{\cC_i}/K_i)$ under this homomorphism is an (open) subgroup of $D_p(\bQ^{\Sigma} \cap \bQ(\cup_{l \in \Sigma}\mu_{l^\infty})/\bQ)$ 
 if and only if $\overline{\p}$ is above $p$.
Thus, $\phi$ satisfies condition $(\Char)$.
Since 
$$\underline{P_{K_i,f}^{\cC_i}} \cap \cs(K_i/\bQ) \supset (\cs(\bQ(\mu_{l})/\bQ)\setminus \{ l \}) \cap \cs(K_i/\bQ) = \cs(K_i(\mu_{l})/\bQ)\setminus \{ l \} \not= \emptyset$$
for $l \in \Sigma$, 
$(ii)$ and $(iii)$ follow from this, together with Lemma \ref{CharDeg} $(i)$ and $(ii)$.
\end{proof}

\begin{rem}\label{incomm7}
We use the notations in Theorem \ref{3.5}.
Assume $\delta_{\sup}(\Sigma) > 0$.
By Theorem \ref{3.5} $(ii)'$ and Proposition \ref{incomm5} $(iii)$, 
$\phi|_{(P_{K_1,f}^{\cC_1} \cap \Sigma(K_1))(K_1^{\cC_1}) \cup (\underline{P_{K_1,f}^{\cC_1}} \cap \cs(K_1/\bQ))(K_1^{\cC_1})}$
induces a local correspondence between 
$$(P_{L_1,f}^{\cC_1} \cap \Sigma(L_1)) \cup (\underline{P_{K_1,f}^{\cC_1}} \cap \cs(K_1/\bQ))(L_1) \text{ and } (P_{L_2,f}^{\cC_2} \cap \Sigma(L_2)) \cup (\underline{P_{K_2,f}^{\cC_2}} \cap \cs(K_2/\bQ))(L_2)$$
for $\overline{\sigma|_{U_1}}$, 
satisfying conditions $(\Char)$ and $(\Deg)$.
This result can be seen as an improvement of Proposition \ref{incomm5} $(iii)$.
However, Proposition \ref{incomm5} $(iii)$ is sufficient for the proof of the main theorem (Theorem \ref{6.1}).
\end{rem}

The results in the rest of this section are a preparation 
for the proof of Proposition \ref{relthm}.
\begin{lem}\label{rel1}
Let $\Sigma \subset P_{\bQ,f}$, $l_0 \in P_{\bQ,f}$ and $p \in P_{\bQ,f} \setminus \Sigma_{l_0}$.
Then 
$D_{p}(\bQ_{\Sigma,l_0}/\bQ) = 1$ 
if and only if 
$\Sigma_{l_0} \subset \cs(\bQ(\mu_{l_0}, \sqrt[l_0]{p})/\bQ)$.
\end{lem}
\begin{proof}
The proof is similar to that of Lemma \ref{incomm1}.
\end{proof}

\begin{lem}\label{rel2}
Let $\Sigma \subset P_{\bQ,f}$, $l_0 \in P_{\bQ,f}$, $n \in \bZ_{\geq 2}$ and $p_1,..., p_n$ be distinct prime numbers in $P_{\bQ,f} \setminus \Sigma_{l_0}$.
Assume $D_{p_1}(\bQ_{\Sigma,l_0}/\bQ) \not= 1$.
Then 
$D_{p_1}(\bQ_{\Sigma,l_0}/\bQ) = \cdots = D_{p_n}(\bQ_{\Sigma,l_0}/\bQ)$ 
if and only if 
there exists $j_i \in \{ 1,...,l_0 - 1 \}$ for each $i \in \{ 2,...,n \}$
such that 
$$\Sigma_{l_0} \subset 
\cs(\bQ(\mu_{l_0}, \sqrt[l_0]{p_1p_2^{-j_2}},..., \sqrt[l_0]{p_1p_n^{-j_n}})/\bQ).$$
\end{lem}
\begin{proof}
By 
the assumption, Lemma \ref{incomm1} and Lemma \ref{rel1}, 
for each $i \in \{ 2,...,n \}$, 
$D_{p_1}(\bQ_{\Sigma,l_0}/\bQ) = D_{p_i}(\bQ_{\Sigma,l_0}/\bQ)$ 
if and only if 
there exists $j_i \in \{ 1,...,l_0 - 1 \}$ such that 
$\Sigma_{l_0} \subset \cs(\bQ(\mu_{l_0}, \sqrt[l_0]{p_1p_i^{-j_i}})/\bQ).$
The assertion follows from this.
\end{proof}

\begin{lem}\label{rel3}
Let $\Sigma \subset P_{\bQ,f}$ and $l_0 \in P_{\bQ,f}$.
For $p \in P_{\bQ,f}$, write $D_{p}$ for the decomposition group of $G(\bQ_{\Sigma,l_0}/\bQ)$ 
at $p$. 
Assume $\delta_{\sup}(\Sigma_{l_0}) > 0$.
Then there exists a finite subset $S \subset P_{\bQ,f} \setminus \Sigma_{l_0}$ (depending on $\Sigma$ and $l_0$) such that the following hold.
\begin{itemize}
\item[$(i)$] 
$D_{p}$ is nontrivial for $p \in P_{\bQ,f} \setminus (\Sigma_{l_0} \cup S)$.

\item[$(ii)$] 
For $p\in P_{\bQ,f} \setminus (\Sigma_{l_0} \cup S)$ and $q\in P_{\bQ,f} \setminus \Sigma_{l_0}$, 
$D_{p} = D_{q}$ if and only if $p=q$.

\end{itemize}
\end{lem}
\begin{proof}
Write 
$S' \defeq \{ p\in P_{\bQ,f} \setminus \Sigma_{l_0} \mid \text{$D_{p}$ is trivial} \}$.
Let $p_1,..., p_n$ be distinct prime numbers in $S'$.
Then, by Lemma \ref{rel1}, we have 
$$\Sigma_{l_0} \subset \cap_{i \in \{ 1,...,n \}} \cs(\bQ(\mu_{l_0}, \sqrt[l_0]{p_i})/\bQ) = \cs(\bQ(\mu_{l_0}, \sqrt[l_0]{p_1},..., \sqrt[l_0]{p_n})/\bQ).$$
Therefore, by the Chebotarev density theorem, 
\begin{equation*}
\begin{split} 
\delta_{\sup}(\Sigma_{l_0}) 
&\leq \delta(\cs(\bQ(\mu_{l_0}, \sqrt[l_0]{p_1},..., \sqrt[l_0]{p_n})/\bQ)) = \frac{1}{(l_0 -1) {l_0}^{n}}.
\end{split}
\end{equation*}
Thus, by the assumption, we obtain $\# S' < \infty$.
For $p,q \in P_{\bQ,f} \setminus (\Sigma_{l_0} \cup S')$, 
we say that $p$ and $q$ are equivalent 
if and only if 
$D_{p} = D_{q}$, 
and write $[p]$ for the equivalence class of $p$.
Write 
$$S'' \defeq \{ p\in P_{\bQ,f} \setminus (\Sigma_{l_0} \cup S') \mid \# [p] \geq 2 \}.$$
By Lemma \ref{rel2} and a similar argument as above, 
we obtain 
$\# S'' < \infty$.
Setting $S \defeq S' \cup S''$, $(i)$ and $(ii)$ hold.
\end{proof}

\begin{prop}\label{relloccor}
We use the notations in Theorem \ref{3.5}.
Let $\Sigma' \subset P_{\bQ,f}$ and $l_0 \in \Sigma$.
Assume that $\delta_{\sup}(\Sigma'_{l_0}) > 0$ and that 
there exists $a \in \bF_{l_0}^{\ast}$ such that the following diagram commutes: 
\begin{equation*}
\xymatrix{
G_{K_1}^{\cC_1}\ar[rr]^-{\simeq}_{\sigma} \ar[d]&&G_{K_2}^{\cC_2}\ar[d]\\
G(\bQ_{\Sigma',l_0}/\bQ)\ar[rr]^-{\simeq}_{(\ )^{a}}&&G(\bQ_{\Sigma',l_0}/\bQ)
}
\end{equation*}
where the vertical arrows are the canonical homomorphisms and the bottom horizontal arrow is the automorphism of $G(\bQ_{\Sigma',l_0}/\bQ)$ obtained by taking $a$-th powers.
Then there 
exist finite subsets $T_1 \subset P_{K_1,f}^{\cC_1}$, $T_2 \subset P_{K_2,f}^{\cC_2}$ 
such that the following hold.
\begin{itemize}
\item[$(i)$] 
For $i=1,2$, write 
\begin{equation*}
\begin{split}
A_i \defeq 
\left\{ \p\in P_{K_i,f}^{\cC_i} \left|
\begin{array}{l}
\text{the image of $D_{\p}(K_i^{\cC_i}/K_i)$ under the canonical homomorphism}\\
\text{$G_{K_i}^{\cC_i} \to G(\bQ_{\Sigma',l_0}/\bQ)$ is nontrivial}\\
\end{array}
\right.\right\}.
\end{split}
\end{equation*}
Then 
$\phi|_{(A_1 \setminus 
T_1
)(K_1^{\cC_1})}$ induces a local correspondence between $(A_1 \setminus 
T_1
)(L_1)$ and $(A_2 \setminus 
T_2
)(L_2)$ for $\overline{\sigma|_{U_1}}$, 
satisfying condition $(\Char)$.

\item[$(ii)$] For $i=1,2$, $\underline{A_i \setminus 
T_i
} \cap \cs(K_i/\bQ) = \underline{P_{K_i,f}^{\cC_i} \setminus T_i}  \cap \cs(K_i/\bQ).$

\end{itemize}
\end{prop}
\begin{proof}
By Proposition \ref{1.2}, Corollary \ref{2.10} and Lemma \ref{2.11}, 
$$\Dec(K_i^{\cC_i}/K_i, A_i(K_i^{\cC_i})) \text{ and } 
\Dec(K_i^{\cC_i}/K_i, (A_i \setminus P_{K_i,l_0})(K_i^{\cC_i}))$$ can be recovered group-theoretically from $G_{K_i}^{\cC_i} \to G(\bQ_{\Sigma',l_0}/\bQ)$ for $i=1,2$.
For $\p\in A_i \setminus P_{K_i,l_0}$, 
by Proposition \ref{1.2} and local class field theory, 
we have 
$$I_\p(K_i^{(l_0)}/K_i) = \Ker(D_\p(K_i^{(l_0)}/K_i) \twoheadrightarrow D_\p(K_i^{(l_0)}/K_i)^{\ab,/\tor}),$$ 
which 
can be recovered group-theoretically.
For $i=1,2$, write 
\begin{equation*}
\begin{split}
B_i \defeq 
\left\{ \p\in A_i \setminus P_{K_i,l_0}
\left|
\begin{array}{l}
\text{the image of $I_\p(K_i^{(l_0)}/K_i)$ 
under 
$G_{K_i}^{(l_0)} \to G(\bQ_{\Sigma',l_0}/\bQ)$ is nontrivial}\\
\end{array}
\right.\right\}.
\end{split}
\end{equation*}
Then 
$\Dec(K_i^{\cC_i}/K_i, (B_i)(K_i^{\cC_i}))$ can be recovered group-theoretically from $G_{K_i}^{\cC_i} \to G(\bQ_{\Sigma',l_0}/\bQ)$, and
$B_i \subset \Sigma'_{l_0}(K_i)$.
For $\p\in B_i$ 
and $p \in \Sigma'_{l_0}$, 
$\p$ is above $p$ 
if and only if 
the image of $I_\p(K_i^{(l_0)}/K_i)$ under 
$$G_{K_i}^{(l_0)} \to G(\bQ_{\Sigma',l_0}/\bQ) 
= \prod_{l \in \Sigma'_{l_0}} G(\bQ(\mu_{l})/\bQ)/G(\bQ(\mu_{l})/\bQ)^{l_0}$$
coincides with the 
$p$-component 
$$G(\bQ(\mu_{p})/\bQ)/G(\bQ(\mu_{p})/\bQ)^{l_0} \subset \prod_{l \in \Sigma'_{l_0}} G(\bQ(\mu_{l})/\bQ)/G(\bQ(\mu_{l})/\bQ)^{l_0}.$$
Since $(\ )^{a} \colon G(\bQ_{\Sigma',l_0}/\bQ) \isom
G(\bQ_{\Sigma',l_0}/\bQ)$
preserves the $p$-component, 
$\phi|_{(B_1)(K_1^{\cC_1})}$
is a local correspondence between 
$B_1$ and $B_2$ for $\sigma$, 
satisfying condition $(\Char)$.
By the Galois equivariance (cf. Theorem \ref{3.5}), 
$\phi$ induces a bijection $\overline{\phi} \colon P_{K_1,f}^{\cC_1} \isom P_{K_2,f}^{\cC_2}$.
Let 
$$T'_1 \defeq (P_{K_1,f}^{\cC_1} \cap(\Sigma'_{l_0}(K_1) \setminus B_1)) \cup \overline{\phi}^{-1}(P_{K_2,f}^{\cC_2}\cap (\Sigma'_{l_0}(K_2) \setminus B_2)) \text{ and } T'_2 \defeq \overline{\phi}(T'_1).$$
Since $\Sigma'_{l_0}(K_i) \setminus B_i \subset \Ram(K_i/\bQ)(K_i)$ for $i=1,2$,  
$T'_1$ and $T'_2$ are finite.
Let $S \subset P_{\bQ,f} \setminus \Sigma'_{l_0}$ be 
as 
in Lemma \ref{rel3}.
Then, by Lemma \ref{rel3}, 
for $\p \in A_i \setminus (B_i \cup T'_i
)$ 
and $p \in P_{\bQ,f} \setminus (\Sigma'_{l_0} \cup S)$, 
$\p$ is above $p$ 
if and only if 
the image of $D_{\p}(K_i^{\cC_i}/K_i)$ 
under 
$G_{K_i}^{\cC_i} \to G(\bQ_{\Sigma',l_0}/\bQ)$ coincides with 
$D_{p}(\bQ_{\Sigma',l_0}/\bQ)$.
Since $(\ )^{a} \colon G(\bQ_{\Sigma',l_0}/\bQ) \isom
G(\bQ_{\Sigma',l_0}/\bQ)$
preserves $D_{p}(\bQ_{\Sigma',l_0}/\bQ)$ for $p \in P_{\bQ,f}$, 
$\phi|_{(A_1 \setminus (B_1 \cup T'_1 \cup S(K_1)))(K_1^{\cC_1})}$
is a local correspondence between 
$A_1 \setminus (B_1 \cup T'_1 \cup S(K_1))$ and $A_2 \setminus (B_2 \cup T'_2 \cup S(K_2))$ for $\sigma$, 
satisfying condition $(\Char)$.
Thus, 
$\phi|_{(A_1 \setminus (T'_1 \cup S(K_1)))(K_1^{\cC_1})}$ is a local correspondence between $A_1 \setminus (T'_1 \cup S(K_1))$ and $A_2 \setminus (T'_2 \cup S(K_2))$ for $\sigma$, 
satisfying condition $(\Char)$.
Setting 
\[
T_i \defeq T'_i \cup
\bigl(P_{K_i,f}^{\cC_i}\cap S(K_i)\bigr)
\qquad (i=1,2),
\]
$(i)$ follows from this, together with Lemma \ref{3.3} $(ii)$, Proposition \ref{1.4} and Lemma \ref{3.4} $(iii)$.
$(ii)$ follows from 
the fact that for $\p \in (P_{K_i,f}^{\cC_i} \setminus T_i)  \cap \cs(K_i/\bQ)(K_i)$, the image of $D_{\p}(K_i^{\cC_i}/K_i)$ under $G_{K_i}^{\cC_i} \to G(\bQ_{\Sigma',l_0}/\bQ)$ is nontrivial.
\end{proof}

\section{Isomorphisms of fields}

In this section, we 
develop two ways 
to show that 
isomorphisms of Galois groups of number fields are induced by field isomorphisms. 
The first is Proposition \ref{U4}, which is essentially based on 
\cite{Uchida2} and will be used in the proof of Theorem \ref{6.1}.
The second is Proposition \ref{4.7}, which is an abstraction of 
\cite[Proposition 2.1]{Shimizu2} 
and will be used in the proof of Proposition \ref{charthm}.

In the rest of this 
paper, 
fix an algebraic closure $\overline{\bQ}$ of $\bQ$, 
and 
suppose that 
all number fields and all algebraic extensions of them 
are subfields of $\overline{\bQ}$.
Write $\widetilde{K}$ for the Galois closure of $K/\bQ$.
Let $K_i$ be a number field for $i=1,2$.

\begin{prop}\label{U1}
Let $\cC_i$ be a nontrivial full class of finite groups for $i=1,2$, $\sigma :G_{K_1}^{\cC_1}\isom G_{K_2}^{\cC_2}$ an isomorphism, and 
$\tau \in \Iso(K_2^{\cC_2}/K_2, K_1^{\cC_1}/K_1)$. Assume that 
$\sigma(H_1) = \tau^{-1} H_1 \tau$ for any open subgroup $H_1$ of $G_{K_1}^{\cC_1}$.
Then 
$\sigma (g_1) =  \tau^{-1} g_1 \tau$ for any $g_1 \in G_{K_1}^{\cC_1}$.
\end{prop}
\begin{proof}
Write $\Sigma \defeq \Sigma(\cC_1) (= \Sigma(\cC_2)$ by Lemma \ref{2.0}).
Write $L$ for the unique $\hat{\bZ}^{\Sigma}$-extension of $\bQ$, which coincides with the composite of the cyclotomic $\bZ_l$-extensions for all $l\in \Sigma$.
By Proposition \ref{3.6}, 
the following diagram commutes:
\begin{equation*}
\xymatrix{
G_{K_1}^{\cC_1}\ar[rr]^-{\simeq}_{\sigma} \ar[rd]_-{}&&G_{K_2}^{\cC_2}\ar[dl]^-{}\\
&G(L/\bQ),
}
\end{equation*}
where the diagonal arrows are the restrictions.
Take any $g_1 \in G_{K_1}^{\cC_1}$ such that  
the closed subgroup $\overline{\langle g_1\rangle}$ of $G_{K_1}^{\cC_1}$ topologically generated by $g_1$ is isomorphic to $\hat{\bZ}^{\Sigma}$ and 
maps injectively to $G(L/\bQ)$.
Note that $G_{K_1}^{\cC_1}$ is topologically generated by such  elements.
By the assumption and a limiting argument, we obtain $\sigma(H_1) = \tau^{-1} H_1 \tau$ for every closed subgroup $H_1$ of $G_{K_1}^{\cC_1}$.
Then 
$$\overline{\langle \sigma (g_1)\rangle} 
= \sigma (\overline{\langle g_1\rangle}) 
= \tau^{-1} \overline{\langle g_1\rangle} \tau = \overline{\langle \tau^{-1} g_1\tau\rangle},$$
and hence 
there exists $n \in \hat{\bZ}^{\Sigma}$
such that $\sigma (g_1) = (\tau^{-1} g_1\tau)^n$.
Thus, we have 
$$g_1|_{L} = \sigma (g_1)|_{L} = (\tau^{-1} g_1\tau)^n|_{L} =(g_1|_{L})^n,$$
so that $n=1$.
\end{proof}

\begin{lem}\label{U3}
Assume that 
$\delta_{\inf}(\cs(K_1/\bQ) \setminus \cs(K_2/\bQ)) = 0$.
Then $K_2 \subset \widetilde{K_1}$.
\end{lem}
\begin{proof}
Since 
$$\cs(K_1/\bQ) \setminus \cs(K_2/\bQ)
= \cs(K_1/\bQ) \setminus (\cs(K_1/\bQ) \cap \cs(K_2/\bQ))
= \cs(\widetilde{K_1}/\bQ) \setminus \cs(\widetilde{K_1}\widetilde{K_2}/\bQ),$$
we have 
\begin{equation*}
\begin{split}
\delta(\cs(K_1/\bQ) \setminus \cs(K_2/\bQ)) 
= \delta(\cs(\widetilde{K_1}/\bQ)) - \delta(\cs(\widetilde{K_1}\widetilde{K_2}/\bQ)) 
=\frac{1}{[\widetilde{K_1}:\bQ]} - \frac{1}{[\widetilde{K_1}\widetilde{K_2}:\bQ]}, 
\end{split} 
\end{equation*}
where the first and second 
equalities follow from \cite[Lemma 4.6]{Shimizu} and the Chebotarev density theorem, respectively.
If $K_2 \not\subset \widetilde{K_1}$, then $\widetilde{K_2} \not\subset \widetilde{K_1}$, 
and hence 
$\delta(\cs(K_1/\bQ) \setminus \cs(K_2/\bQ)) >0$.
\end{proof}

\begin{prop}\label{U4}
For $i=1,2$, let $\cC_i$ be a nontrivial full class of finite groups, $\sigma :G_{K_1}^{\cC_1}\isom G_{K_2}^{\cC_2}$ an isomorphism, and $S_0\subset P_{\bQ,f}$.
Write $\Sigma \defeq \Sigma(\cC_1) (= \Sigma(\cC_2)$ by Lemma \ref{2.0}).
Assume that the following conditions hold:
\begin{itemize}
\item[(a)]
$\# \Sigma = \infty$.

\item[(b)]
$\delta_{\inf}(\cs(K_1/\bQ) \setminus S_0) =0$.

\item[(c)]
There exists a 
local correspondence between $S_0(K_1)$ and $S_0(K_2)$ for $\sigma$ satisfying condition $(\Char)$.

\end{itemize}
Then 
there exists 
a unique 
$\tau\in \Iso(K_2^{\cC_2}/K_2, K_1^{\cC_1}/K_1)$
such that 
$\sigma (g_1) =  \tau^{-1} g_1 \tau$ for any $g_1 \in G_{K_1}^{\cC_1}$.
\end{prop}
\begin{proof}We first reduce to the case where the local correspondence
in (c) also satisfies condition $(\Deg)$.
By (b) and the Chebotarev density theorem, we have
\[
S_0\cap\cs(K_1K_2/\bQ)\ne\emptyset.
\]
Applying Lemma \ref{CharDeg} with $S_i=S_0(K_i)$
and $U_i=G_{K_i}^{\cC_i}$, we obtain
\[
S_0'\defeq S_0\cap\cs(K_1/\bQ)
=
S_0\cap\cs(K_2/\bQ),
\]
and the restricted local correspondence between
$S_0'(K_1)$ and $S_0'(K_2)$ satisfies
conditions $(\Char)$ and $(\Deg)$.
Moreover,
\[
\cs(K_1/\bQ)\setminus S_0'
=
\cs(K_1/\bQ)\setminus S_0.
\]
Thus, replacing $S_0$ by $S_0'$, we may assume that
the local correspondence in (c) also satisfies
condition $(\Deg)$.

The uniqueness of $\tau$ follows from Proposition \ref{1.5}. Let us prove the existence of $\tau$.
By Proposition \ref{U1}, we are reduced to proving 
\begin{equation*}
B \defeq 
\left\{ \tau \in \Iso(K_2^{\cC_2}/K_2, K_1^{\cC_1}/K_1) \left|
\begin{array}{l}
\sigma(H_1) = \tau^{-1} H_1 \tau 
\text{
for any open subgroup $H_1$ of $G_{K_1}^{\cC_1}$}
\end{array}
\right.\right\} \not= \emptyset.
\end{equation*}
The rest of the proof is based on the proof of \cite[THEOREM]{Uchida2}.
Take any open normal subgroup $U_1$ 
of $G_{K_1}^{\cC_1}$.
Set $U_2 \defeq \sigma(U_1)$.
For $i=1,2$, write $K_i'$ for the finite Galois subextension of $K_i^{\cC_i}/K_i$ corresponding to $U_i$.
Let $K/\bQ$ be a finite Galois extension such that $K\supset K_1'K_2'$.
Let $\{ U_1 = U_{11},\dots,U_{1m} \}$ be the set of all open subgroups of $G_{K_1}^{\cC_1}$ containing $U_1$.
Set $U_{2j} \defeq \sigma(U_{1j})$ for $j \in \{ 1,\dots,m \}$.
Write $F_{ij}$ 
for the finite subextension of $K_i'/K_i$ corresponding to $U_{ij}$, and $N_{ij} \defeq G(K/F_{ij})$.
Note that $F_{i1}=K_i'$ and
\[
N_{i1}=G(K/K_i')\triangleleft N_{ij}=G(K/F_{ij}),
\qquad
N_{ij}/N_{i1}\simeq G(K_i'/F_{ij})\simeq U_{ij}/U_i.
\]
The relevant intermediate fields are arranged as follows:
\[
\xymatrix@C=3.5em@R=1.7em{
  & K & \\
  K_1' \ar@{-}[ur]^{N_{11}}
  &&
  K_2' \ar@{-}[ul]_{N_{21}}
  \\
  F_{1j} \ar@{-}[u]^{N_{1j}/N_{11}}
  &&
  F_{2j} \ar@{-}[u]_{N_{2j}/N_{21}}
  \\
  K_1 \ar@{-}[u]
  &&
  K_2 \ar@{-}[u]
  \\
  & \bQ \ar@{-}[ul] \ar@{-}[ur] &
}
\]
Here and below, lines in field diagrams indicate inclusions,
with larger fields placed above smaller ones. A label on a line denotes the Galois group of the corresponding extension.
Set 
\begin{equation*}
\begin{split} 
T'(U_1) 
&\defeq 
\left\{ \tau \in G_\bQ \left|
\begin{array}{l}
\tau(F_{2j})=F_{1j}
\text{ for all }
j \in \{ 1,\dots,m \}
\end{array}
\right.\right\}\\
&=
\left\{ \tau \in G_\bQ \left|
\begin{array}{l}
N_{2j} = \tau^{-1} N_{1j} \tau 
\text{ for all }
j \in \{ 1,\dots,m \}
\end{array}
\right.\right\}.
\end{split} 
\end{equation*}
To prove $B\not= \emptyset$, 
it suffices to show that 
$T'(U_1) \neq \emptyset$ 
for any $U_1$.
Indeed, having shown this, 
we obtain $\cap_{U_1} T'(U_1) \neq \emptyset$ by the compactness of $G_\bQ$, and the restriction to  $K_2^{\cC_2}$ of any isomorphism in $\cap_{U_1} T'(U_1)$ induces an isomorphism in $B$.
Let $H \defeq G(K/\bQ)$.
Take $p\in \Sigma$ such that $p> [K:\bQ]$, whose existence is guaranteed by (a).\footnote{In the proof of \cite[THEOREM]{Uchida2}, we have to take $p\in P_{\bQ,f}$ such that $p \equiv 1 \mod [K:\bQ]$. However, since $\delta(\{ p \in P_{\bQ,f} \mid p \equiv 1 \mod [K:\bQ] \}) = \delta(\cs(\bQ(\mu_{[K:\bQ]})/\bQ))=1/[\bQ(\mu_{[K:\bQ]}):\bQ]$, 
the existence of such $p$ in $\Sigma$ is not guaranteed by (a). 
In order to avoid this problem, we modify the proof of \cite[THEOREM]{Uchida2} by using Proposition \ref{U1}.}
Note that $\#H=[K:\bQ]<p$.
In particular, the order of every subgroup of $H$ is prime to $p$. We will use this fact to describe the $p$-group quotients below, and to compare coefficients at the end of the proof.
Let 
$$A \defeq \bF_p[H]u_1\oplus\dots\oplus\bF_p[H]u_m
$$ 
be a left $H$-module isomorphic to a direct sum of $m$ copies of the group ring $\bF_p[H]$, where $\{ u_1,\dots,u_m\}$ is a free basis.
Let 
$E \defeq A \rtimes H$ be the semidirect product of $H$ acting on $A$ as above.
By \cite[(9.2.9) Proposition]{NSW}, 
the surjection $G_\bQ \twoheadrightarrow H$ can be lifted to a surjection $G_\bQ \twoheadrightarrow E$.
Write $L$ for the finite Galois extension of $\bQ$ corresponding to $G_\bQ \twoheadrightarrow E$, 
and 
we will identify $E$ with $G(L/\bQ)$ via this surjection.
Write $L_j$ for the finite Galois subextension of $L/K$ corresponding to the projection $A \twoheadrightarrow \bF_p[H]u_j$.
Put $V_j\defeq\bF_p[H]u_j$. Then $L_j/\bQ$ is Galois, and 
\[
A=\bigoplus_{j=1}^m V_j,\qquad
G(L_j/K)=V_j,\qquad
G(L_j/K_i')=V_j\rtimes N_{i1},\qquad
G(L_j/\bQ)=V_j\rtimes H.
\]
Let $M_{1j}$ be the maximal abelian $p$-extension of $K_1'$ contained in $L_j$ such that the action of $N_{1j}/N_{11}$ on $G(M_{1j}/K_1')$ is trivial.
As $M_{1j}$ is a subextension of $K_1^{\cC_1}/K_1$, 
a subextension $M_{2j}$ of $K_2^{\cC_2}/K_2$ corresponds to $M_{1j}$ under $\sigma$.
By (c), Proposition \ref{1.4}, Lemma \ref{3.4} $(iii)$ 
and Lemma \ref{U2}, 
we have 
$$[F_{1j}:\bQ]=[F_{2j}:\bQ], \qquad
S_0 \cap \cs(K_1/\bQ)=S_0 \cap \cs(K_2/\bQ), \qquad
S_0 \cap \cs(M_{1j}/\bQ)=S_0 \cap \cs(M_{2j}/\bQ).$$
Therefore,
\[
\cs(M_{1j}/\bQ)\setminus\cs(M_{2j}/\bQ)
\subset
\cs(K_1/\bQ)\setminus S_0.
\]
By (b) and Lemma \ref{U3}, we obtain
$M_{2j}\subset\widetilde{M_{1j}}\subset L_j$.
Let $M_{2j}'$ be 
the maximal abelian $p$-extension of $K_2'$ contained in $L_j$ such that the action of $N_{2j}/N_{21}$ on $G(M_{2j}'/K_2')$ is trivial.
Then $M_{2j} \subset M_{2j}'$. Let $(M_{1j}\subset )M_{1j}'$ be the subextension of $K_1^{\cC_1}/K_1$ corresponding to $M_{2j}'$ under $\sigma$. Then, by a similar argument as above, we obtain $M_{1j}'\subset L_j$. Therefore, by the maximality of $M_{1j}$, we have $M_{1j}= M_{1j}'$, and hence $M_{2j} = M_{2j}'$.

We now describe $G(M_{ij}/K_i')$ as a quotient of $G(L_j/K_i')$ for $i=1,2$.
Since $p\nmid\#N_{i1}$, every $p$-group quotient of
$V_j\rtimes N_{i1}$ kills $N_{i1}$.
Put $B_{ij}\defeq G(L_j/KM_{ij})\subset V_j$.
Using the triviality of the action of $N_{ij}/N_{i1}$ on $G(M_{ij}/K_i')$ and the maximality of $M_{ij}$,
we obtain
\[
B_{ij}=\sum_{s\in N_{ij}}(s-1)V_j,
\qquad
G(M_{ij}/K_i')\simeq V_j/B_{ij}.
\]
We identify $V_j/B_{ij}$ with $\bF_p[N_{ij}\backslash H]u_j$ via $g u_j+B_{ij}\mapsto[N_{ij}g]u_j$, 
where $\bF_p[N_{ij}\backslash H]$ denotes the $\bF_p$-vector space with basis $N_{ij}\backslash H$.
Equivalently, restriction gives the commutative diagram
\[
\xymatrix@C=4em@R=2.2em{
  G(L_j/K_i')
    \ar@{->>}[r]^{\res}
    \ar@{=}[d]
  &
  G(M_{ij}/K_i')\ar[d]^{\sim}
  \\
  V_j\rtimes N_{i1}
    \ar@{->>}[r]_{q_{ij}}
  &
  V_j/B_{ij}.
}
\]
Here $q_{ij}(v,n)=v+B_{ij}$, and its kernel is $B_{ij}\rtimes N_{i1}$.
Since $[K:K_i']$ is prime to $p$ and $[M_{ij}:K_i']$
is a power of $p$, we have $K\cap M_{ij}=K_i'$.
Consequently, restriction identifies $G(KM_{ij}/K)$
with $G(M_{ij}/K_i')$.
The fields and Galois groups fit into the following diagram:
\[
\xymatrix@C=3.5em@R=1.7em{
  & L_j & \\
  & KM_{ij}\ar@{-}[u]^{B_{ij}} & \\
  K\ar@{-}[ur]^{V_j/B_{ij}}
  &&
  M_{ij}\ar@{-}[ul]_{N_{i1}}
  \\
  & K_i'
    \ar@{-}[ul]^{N_{i1}}
    \ar@{-}[ur]_{V_j/B_{ij}} &
  \\
  & F_{ij}\ar@{-}[u]_{N_{ij}/N_{i1}} &
}
\]
Finally, set
\[
M_i\defeq\prod_{j=1}^m M_{ij},
\qquad
A_i\defeq\bigoplus_{j=1}^m B_{ij}\subset A.
\]
Then $G(L/KM_i)=A_i$.

Now, let us prove that 
any element $a \in A_1$ is conjugate to some element of $A_2$ in $E$.
By (b), we have $\delta_{\sup}((S_0 \cap \cs(K/\bQ))(K)) =1$.
Therefore, by the Chebotarev density theorem for the Galois extension $L/K$, there exists $\p \in (S_0 \cap \cs(K/\bQ))(L)$ such that $a= \Frob_\p$.
Let $l \defeq \p|_{\bQ} \in S_0$.
Then 
$l \in \cs(K/\bQ)$ and the  local degree of $\p|_{M_1}$ is $1$.
By Proposition \ref{1.4} and Lemma \ref{3.4} $(iii)$, 
the local correspondence in (c) induces a bijection $\phi \colon S_0(M_1) \isom S_0(M_2)$, preserving the residue characteristics and the local degrees.
Take $e \in E$ such that $e\p|_{M_2} = \phi(\p|_{M_1})$. 
Since $e\p|_{KM_2}$ is above $l$ and $\phi(\p|_{M_1})$, 
the local degree of $e\p|_{KM_2}$ is $1$, and hence 
$e^{-1} ae = \Frob_{e\p} \in A_2$.
As $A_1, A_2 \subset A$, 
writing $h$ for the image of $e^{-1}$ under $E \twoheadrightarrow H$, we have $ha \in A_2$.

We now put 
$$a \defeq \sum_{j=1}^{m} \sum_{s\in N_{1j}}(s - 1)u_j \in A_1.$$
There exists $h\in H$ such that $ha \in A_2$, i.e., 
$$h \sum_{s\in N_{1j}}(s - 1) \in \sum_{t\in N_{2j}}(t - 1)\bF_p[H] \text{ for }j \in \{ 1,\dots,m \}.$$
Hence 
$$\sum_{t\in N_{2j}}th \sum_{s\in N_{1j}}(s - 1) =0 \text{ for }j \in \{ 1,\dots,m \}.$$
Therefore, 
\begin{equation}\label{eq1}
\sum_{t\in N_{2j}} \sum_{s\in N_{1j}}ths 
= \# N_{1j}\sum_{t\in N_{2j}}th
\text{ for }j \in \{ 1,\dots,m \}.
\end{equation}
Fix $j\in\{1,\dots,m\}$ and $s_0\in N_{1j}$. We compare the coefficients of $hs_0$ in (\ref{eq1}).
We have 
\begin{equation*}
1 \leq
\# \left\{ (t, s) \in N_{2j}\times N_{1j} \left|
\begin{array}{l}
ths = hs_0
\end{array}
\right.\right\} 
\leq \# N_{2j}
\leq 
[K:\bQ]
<p
\end{equation*}
since $(1, s_0)$ is contained in this set and for each $t\in N_{2j}$, 
at most one element $s\in N_{1j}$ satisfies $ths = hs_0$.
Therefore, the coefficient of $hs_0$ in the left side of (\ref{eq1}) is not $0$.
Hence there exists $t \in N_{2j}$ such that $th = hs_0$.
This shows $hN_{1j}h^{-1} \subset N_{2j}$, 
hence $hN_{1j}h^{-1} = N_{2j}$ 
as they have the same order.
Thus, any lift of $h^{-1}$ under $G_\bQ \twoheadrightarrow H$ is contained in $T'(U_1)$.
\end{proof}

For $\tau\in \Iso(K_2, K_1)$ and a prime number $l$, we write 
$\widetilde{\tau}^\ast_{l}:\Gamma_{K_1, l} \isom \Gamma_{K_2, l}$ for the isomorphism induced by $\tau$.

The following proposition is an abstraction of \cite[Proposition 3.3]{Shimizu}.
\begin{prop}\label{4.1}
Let $l$ be a prime number, 
$S_i \subset P_{K_i,f}$ a set of nonarchimedean primes of $K_i$ for $i=1,2$  
and $\sigma :\Gamma_{K_1, l} \isom \Gamma_{K_2, l}$ an isomorphism.
Assume that the following conditions hold:
\begin{itemize}
\item[(a)]
There exists a weak local correspondence between $S_1$ and $S_2$ for $\sigma$ satisfying condition $(\Char)$ and condition $(\Frob)$.

\item[(b)]
There exists a finite extension $L/K_1K_2$ such that 
$L/\bQ$ is Galois and 
$\delta_{\sup}(S_1(L)) > 0$.

\end{itemize}
Then there exists $\tau \in G(L/\bQ)$ such that 
$\sigma \circ \pi_{L/K_1, l} = \widetilde{(\tau|_{K_2})}^\ast_{l} \circ \pi_{L/\tau(K_2), l}$.
\end{prop}
\begin{proof}
Write $\phi\colon S_1(K_1^{(\infty,l)}) \to S_2(K_2^{(\infty,l)})$ for a weak local correspondence in (a).
We have the following diagram:
$$
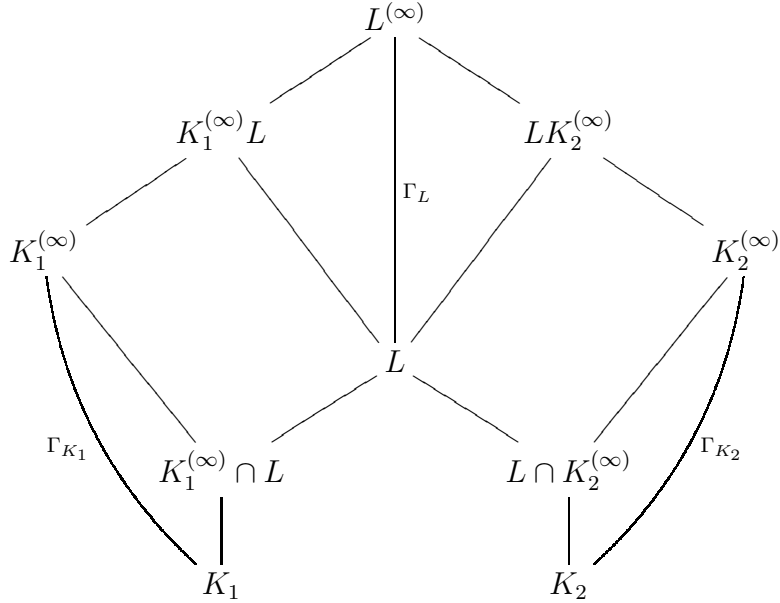
\xymatrix{
& & L^{(\infty)} \ar@{-}[ld] \ar@{-}[ddd]^{\Gamma_{L}} \ar@{-}[rd] & & \\
& K_1^{(\infty)}L \ar@{-}[ld] \ar@{-}[rdd] & & LK_2^{(\infty)} \ar@{-}[ldd] \ar@{-}[rd] & \\
K_1^{(\infty)} \ar@{-}[rdd]^{} \ar@/_15pt/@{-}[rddd]_{\Gamma_{K_1}} & & & & K_2^{(\infty)} \ar@{-}[ldd]_{} \ar@/^15pt/@{-}[lddd]^{\Gamma_{K_2}} \\
& & L \ar@{-}[ld] \ar@{-}[rd] & & \\
& K_1^{(\infty)} \cap L \ar@{-}[d] & & L \cap K_2^{(\infty)} \ar@{-}[d] & \\
& K_1 & & K_2 & \\
}
$$
For 
$\tau \in G(L/\bQ)$, 
set 
$\psi_\tau 
:\Gamma_{L} \to \Gamma_{K_2}$, 
$\gamma \mapsto (\sigma \circ \pi_{L/K_1,l})(\gamma) \cdot (\pi_{L/K_2,l} \circ \widetilde{\tau}^\ast_{l})(\gamma)^{-1}$. 
Then this is a homomorphism of free $\bZ_l$-modules.

Assume that for any $\tau \in G(L/\bQ)$, 
$\sigma \circ \pi_{L/K_1,l} \neq \pi_{L/K_2,l} \circ \widetilde{\tau}^\ast_{l}$.
Then 
$\rank_{\bZ_l}\Ker(\psi_\tau) < r_l(L)$, so that 
$\Ker(\psi_\tau)$ has Haar measure $0$ in $\Gamma_{L}$.
Hence 
$\bigcup_{\tau \in G(L/\bQ)} \Ker(\psi_\tau)$ also has Haar measure $0$.
By (b), we have 
$\delta_{\sup}(S_1(L) \cap \cs(L/\bQ)(L)) = \delta_{\sup}(S_1(L)) > 0$.
Then, 
by the Chebotarev density theorem for infinite extensions 
(\cite[Chapter I, 2.2, COROLLARY 2]{Serre}), 
there exists 
$\p \in (S_1 \setminus P_{K_1,l})(L) \cap \cs(L/\bQ)(L)$ 
such that 
$\Frob_\p \notin \bigcup_{\tau \in G(L/\bQ)} \Ker(\psi_\tau)$.
For $i=1,2$, 
we set 
$\p_i \defeq \p|_{K_i}$, 
then 
$\pi_{L/K_i,l}(\Frob_{\p})=\Frob_{\p_i}$.
Take $\bap_1 \in P_{K_1^{(\infty)},f}$ above $\p_1$.
Since $\phi$ satisfies condition $(\Frob)$, 
we obtain $\sigma(\Frob_{\p_1})=\Frob_{\phi(\bap_1)|_{K_2}}$. 
Set $p \defeq \p|_{\bQ}$.
By condition $(\Char)$, 
$\phi(\bap_1)$ is also above $p$.
Therefore, there exists $\tau' \in G(L/\bQ)$ such that 
$\phi(\bap_1)|_{K_2}=(\tau' \cdot \p)|_{K_2}$.
Then we have $\Frob_{\phi(\bap_1)|_{K_2}} = \pi_{L/K_2,l} \circ \widetilde{\tau'}^\ast_{l} (\Frob_{\p})$.
Thus, we obtain 
\begin{equation*}
\begin{split}
\sigma \circ \pi_{L/K_1,l}(\Frob_\p)
=\Frob_{\phi(\bap_1)|_{K_2}}
=\pi_{L/K_2,l} \circ \widetilde{\tau'}^\ast_{l} (\Frob_{\p}).
\end{split}
\end{equation*}
Namely, 
$\Frob_\p \in \Ker(\psi_{\tau'})$. 
However, this contradicts the fact that 
$\Frob_\p \notin \bigcup_{\tau \in G(L/\bQ)} \Ker(\psi_\tau)$.

Therefore, 
there exists $\tau \in G(L/\bQ)$ such that 
$\sigma \circ \pi_{L/K_1,l} = \pi_{L/K_2,l} \circ \widetilde{\tau}^\ast_{l} = \widetilde{(\tau|_{K_2})}^\ast_{l} \circ \pi_{L/\tau(K_2),l}$.
\end{proof}

To obtain field isomorphisms, we collect the author's previous results.

\begin{lem}[{\cite[Lemma 3.4]{Shimizu}}]\label{4.4}
Let $l$ be a prime number.
Assume that 
there exist a finite extension $L/K_1K_2$ 
and an isomorphism $\overline{\sigma}: \Gamma_{K_1, l} \isom \Gamma_{K_2, l}$
such that 
$\overline{\sigma} \circ \pi_{L/K_1, l} = \pi_{L/K_2, l}$.
Then 
$K_1^{(\infty,l)}K_2 = K_1K_2^{(\infty,l)}$.
\end{lem}

\begin{prop}[{\cite[Proposition 3.5]{Shimizu}}]\label{4.5}
Assume that the following conditions hold:
\begin{itemize}
\item[(a)]
$K_i$ has a complex prime 
for one $i$. 

\item[(b)]
There exists a finite extension $L/K_1K_2$ such that 
$L/\bQ$ is Galois and 
$K_1^{(\infty,l)}L = LK_2^{(\infty,l)}$ for a prime number $l$.

\end{itemize}
Then 
$K_1 = K_2$.
\end{prop}

\begin{lem}[{\cite[Lemma 1.1]{Shimizu2}}]\label{4.6}
Let $L/K$ be a finite Galois extension, and $l$ a prime number.
Assume that $L$ has a complex prime. 
Then the canonical action of $G(L/K)$ on $\Gamma_{L, l}$ induced by conjugation is faithful.
\end{lem}

\begin{prop}\label{4.7}
For $i=1,2$, let $M_i/K_i$ be a Galois extension, $\sigma :G(M_1/K_1)\isom G(M_2/K_2)$ an isomorphism, and 
$U_i$ an open normal subgroup of $G(M_i/K_i)$ with $\sigma(U_1) = U_2$.
For $i=1,2$, write $L_i$ for the finite Galois subextension of $M_i/K_i$ corresponding to $U_i$.
Let $l$ be a prime number and $S_i \subset P_{L_i,f}$ a set of nonarchimedean primes of $L_i$ for $i=1,2$.
Assume that the following conditions hold:
\begin{itemize}

\item[(a)]
$L_{i}^{(\infty,l)} \subset M_i$ for $i=1,2$.

\item[(b)]
There exists a weak local correspondence between $S_1$ and $S_2$ for the isomorphism $\Gamma_{L_1, l} \isom \Gamma_{L_2, l}$ induced by $\sigma|_{U_1}$ 
satisfying condition $(\Char)$ and condition $(\Frob)$.

\item[(c)]
There exists a finite extension $L/L_1L_2$ such that 
$L/\bQ$ is Galois and 
$\delta_{\sup}(S_1(L)) > 0$.

\item[(d)]
$L_i$ has a complex prime for one $i$.

\end{itemize}
Then 
there exists $\tau \in G(L/\bQ)$ such that 
$\tau|_{L_2} \in \Iso(L_2/K_2, L_1/K_1)$ and the isomorphism $G(L_1/K_1) \isom G(L_2/K_2)$ induced by $\sigma$ coincides with the isomorphism induced by $\tau|_{L_2}$.
In particular, for any open normal subgroups $U_1'$, $U_2'$ of $G(M_1/K_1)$,  $G(M_2/K_2)$ containing $U_1$, $U_2$, respectively, with $\sigma(U_1') = U_2'$, it follows that 
$\tau|_{L_2'} \in \Iso(L_2'/K_2, L_1'/K_1)$
and the isomorphism $G(L_1'/K_1) \isom G(L_2'/K_2)$ induced by $\sigma$ coincides with the isomorphism induced by $\tau|_{L_2'}$, where $L_1'$, $L_2'$ are finite Galois subextensions of $L_1/K_1$, $L_2/K_2$, corresponding to $U_1'$, $U_2'$, respectively.
\end{prop}
\begin{proof}
Write $\overline{\sigma}_{L_1}:G(L_1/K_1) \isom G(L_2/K_2)$ for the isomorphism induced by $\sigma$, and $\widetilde{\sigma}_{L_1, l}:\Gamma_{L_1, l} \isom \Gamma_{L_2, l}$ for the isomorphism induced by $\sigma|_{U_1}$. 
By Proposition \ref{4.1}, 
there exists $\tau \in G(L/\bQ)$ such that 
$\widetilde{\sigma}_{L_1, l} \circ \pi_{L/L_1, l} = \widetilde{(\tau|_{L_2})}^\ast_{l} \circ \pi_{L/\tau(L_2), l}$.
Then we have $\widetilde{(\tau|_{L_2})}_{l}^{\ast -1} \circ \widetilde{\sigma}_{L_1, l} \circ \pi_{L/L_1, l} = \pi_{L/\tau(L_2), l}$.
By Lemma \ref{4.4}, 
we obtain 
$L_1^{(\infty,l)}\tau(L_2) = L_1\tau(L_2)^{(\infty,l)}$, 
and hence 
$L_1^{(\infty,l)}L = L\tau(L_2)^{(\infty,l)}$.
By Proposition \ref{4.5}, 
we obtain 
$\tau|_{L_2} \in \Iso(L_2, L_1)$.
Further, since $\Image(\pi_{L/L_1, l})$ is open in $\Gamma_{L_1, l}$, 
$\widetilde{\sigma}_{L_1, l}$ coincides with $\widetilde{(\tau|_{L_2})}^\ast_{l}$.

Write $(\tau|_{L_2})^\ast: \Aut(L_1) \isom \Aut(L_2)$ for the isomorphism induced by $\tau|_{L_2}$.
By the 
equality $\widetilde{\sigma}_{L_1, l} = \widetilde{(\tau|_{L_2})}^\ast_{l}$, 
for any $\tau_1 \in G(L_1/K_1)$, 
the actions of $\overline{\sigma}_{L_1}(\tau_1)$ and $(\tau|_{L_2})^\ast(\tau_1)$ on $\Gamma_{L_2, l}$ coincide.
Further, by (d) and Lemma \ref{4.6}, the conjugation action of $\Aut(L_2)$
on $\Gamma_{L_2, l}$ is faithful.
Therefore, 
we obtain $\overline{\sigma}_{L_1}(\tau_1) = (\tau|_{L_2})^\ast(\tau_1)$, 
so that $\overline{\sigma}_{L_1} = (\tau|_{L_2})^\ast|_{G(L_1/K_1)}$ in $\Iso(G(L_1/K_1), G(L_2/K_2))$.
Hence 
$\tau|_{L_2}$ is compatible with the actions of 
$G(L_2/K_2)$ and $G(L_1/K_1)$ 
on $L_2$ and $L_1$.
Thus, 
for any $U_1'$, $U_2'$ 
as in the assertion, 
$\tau|_{L_2}$ induces an isomorphism between $L_2^{G(L_2/L_2')} = L_2'$ and $L_1^{G(L_1/L_1')} = L_1'$.
Therefore, we obtain 
$\tau|_{L_2'} \in \Iso(L_2'/K_2, L_1'/K_1)$, 
in particular, $\tau|_{L_2} \in \Iso(L_2/K_2, L_1/K_1)$.
Further, the isomorphism $G(L_1'/K_1) \isom G(L_2'/K_2)$ induced by $\sigma$ 
coincides with the isomorphism induced by $\tau|_{L_2'}$.
\end{proof}

\section{The calculation of the Dirichlet density}
In this section, we calculate the Dirichlet density of 
$\underline{P_{K,f}^{\cC}}$ 
as a preparation for Theorem \ref{6.1}.

\begin{prop}\label{5.2}
Let $\Sigma \subset P_{\bQ,f}$.
If $\sum_{p \in \Sigma} p^{-1} = \infty$, then $\delta(\cup_{p \in \Sigma} \cs(K(\mu_p)/K))=1$.
\end{prop}
\begin{proof}
Removing a finite subset from $\Sigma$, we may assume that $K$ and $\bQ(\cup_{p \in \Sigma}\mu_p)$ are linearly disjoint over $\bQ$ and that $2\notin\Sigma$.
By {\cite[Lemma 4.6]{Shimizu}}, 
we have 
$$\delta_{\sup}(P_{K,f} \setminus \cup_{p \in \Sigma} \cs(K(\mu_p)/K)) + \delta_{\inf}(\cup_{p \in \Sigma} \cs(K(\mu_p)/K)) = 1.$$
Hence 
we are reduced to showing that 
$\delta_{\sup}(P_{K,f} \setminus \cup_{p \in \Sigma} \cs(K(\mu_p)/K))=0$.
Let $\Sigma' \subset \Sigma$ be any finite subset.
Then 
$G(K(\cup_{p \in \Sigma'}\mu_p)/K) \simeq \prod_{p \in \Sigma'}G(K(\mu_p)/K) \simeq \prod_{p \in \Sigma'}G(\bQ(\mu_p)/\bQ)$.
For $p \in \Sigma'$, write $\pr_p \colon G(K(\cup_{p \in \Sigma'}\mu_p)/K) \twoheadrightarrow G(K(\mu_p)/K)$ for 
the canonical surjection.
Then 
$P_{K,f} \setminus (\cup_{p \in \Sigma'} \cs(K(\mu_p)/K) \cup \Ram(K(\cup_{p \in \Sigma'}\mu_p)/K))$ coincides with 
$$\{ \p\in P_{K,f} \mid \text{$\p$ is unramified in $K(\cup_{p \in \Sigma'}\mu_p)/K$ and $\pr_p(\Frob_\p)$ is non-trivial for all $p \in \Sigma'$} \},$$
where $\Frob_\p$ is the Frobenius element in $G(K(\cup_{p \in \Sigma'}\mu_p)/K)$ at $\p$.
Therefore, by the Chebotarev density theorem, 
\begin{equation*}
\begin{split} 
\delta(P_{K,f} \setminus \cup_{p \in \Sigma'} \cs(K(\mu_p)/K)) 
&= \frac{ \# \{ \tau \in G(K(\cup_{p \in \Sigma'}\mu_p)/K) \mid \text{$\pr_p(\tau)$ is non-trivial for all $p \in \Sigma'$} \} }{ \# G(K(\cup_{p \in \Sigma'}\mu_p)/K)}\\
&=\frac{ \prod_{p \in \Sigma'} (p-2) }{ \prod_{p \in \Sigma'} (p-1) }\\
&=\prod_{p \in \Sigma'} \frac{p-2}{p-1}.
\end{split}
\end{equation*}
Hence 
$\delta_{\sup}(P_{K,f} \setminus \cup_{p \in \Sigma} \cs(K(\mu_p)/K)) 
\leq \delta(P_{K,f} \setminus \cup_{p \in \Sigma'} \cs(K(\mu_p)/K)) 
= \prod_{p \in \Sigma'} \frac{p-2}{p-1}$.
By passing to the limit, 
we obtain 
$\delta_{\sup}(P_{K,f} \setminus \cup_{p \in \Sigma} \cs(K(\mu_p)/K)) 
\leq \prod_{p \in \Sigma} \frac{p-2}{p-1}$.
Further, 
$\prod_{p \in \Sigma} \frac{p-2}{p-1} =0 \Leftrightarrow \prod_{p \in \Sigma} \frac{p-1}{p} = 0
$.
Thus, the assertion follows from Lemma \ref{incomm3} $(ii)$.
\end{proof}

\begin{rem}\label{5.3}
Let $\Sigma \subset P_{\bQ,f}$.
Assume that $K$ and
$\bQ(\cup_{p\in\Sigma}\mu_p)/\bQ$
are linearly disjoint and that $2\notin\Sigma$.
Then
\[
\delta\left(
P_{K,f}\setminus
\bigcup_{p\in\Sigma}\cs(K(\mu_p)/K)
\right)
=
\prod_{p\in\Sigma}\frac{p-2}{p-1}.
\]
Indeed, if $\sum_{p\in\Sigma}p^{-1}=\infty$,
the equality follows from Proposition \ref{5.2}.
Otherwise, for any finite subset $\Sigma'\subset\Sigma$,
the Brun--Titchmarsh inequality
(cf. \cite[Theorem 2]{Montgomery-Vaughan}) gives
\[
\delta_{\sup}\left(
\bigcup_{p\in\Sigma\setminus\Sigma'}
\cs(K(\mu_p)/K)
\right)
\leq
2[K:\bQ]
\sum_{p\in\Sigma\setminus\Sigma'}\frac1{p-1}.
\]
Thus, the equality follows from the Chebotarev density
theorem for finite extensions by letting
$\Sigma'$ increase to $\Sigma$.
In particular, the converse of Proposition \ref{5.2}
holds under these assumptions.

More generally, 
without the above assumptions,
we can prove that 
$\delta(\cup_{p \in \Sigma} \cs(K(\mu_p)/K))=1$ 
if and only if $\sum_{p \in \Sigma} p^{-1} = \infty$ or 
there exists a finite subset $\Sigma' \subset \Sigma$ such that $\delta(\cup_{p \in \Sigma'} \cs(K(\mu_p)/K))=1$.
Indeed, choose a finite subset $\Sigma'\subset\Sigma$
such that $2\notin\Sigma\setminus\Sigma'$ and
\[
K\cap\bQ(\cup_{p\in\Sigma}\mu_p)
\subset
\bQ(\cup_{p\in\Sigma'}\mu_p).
\]
Then restriction gives an isomorphism
\[
G(K(\cup_{p\in\Sigma}\mu_p)/K)
\simeq
G(K(\cup_{p\in\Sigma\setminus\Sigma'}\mu_p)/K)
\times
G(K(\cup_{p\in\Sigma'}\mu_p)/K),
\]
and hence 
we obtain 
$$\delta(P_{K,f} \setminus \cup_{p \in \Sigma} \cs(K(\mu_p)/K)) = 
\delta(P_{K,f} \setminus \cup_{p \in \Sigma\setminus\Sigma'} \cs(K(\mu_p)/K))
 \times 
\delta(P_{K,f} \setminus \cup_{p \in \Sigma'} \cs(K(\mu_p)/K))$$
by the Chebotarev density theorem 
and the same approximation argument as above.
\end{rem}

\begin{lem}\label{5.4.5}
Let $L/K$ be a  finite extension, 
$S\subset P_{K,f}$.
If $\delta_{\sup}(S) = 1$ (resp. $\delta(S) = 1$), then $\delta_{\sup}(S(L)) = 1$ (resp. $\delta(S(L)) = 1$).
\end{lem}
\begin{proof}
Assume $\delta_{\sup}(S) = 1$ (resp. $\delta(S) = 1$). Then, by \cite[Lemma 4.6]{Shimizu}, 
$\delta_{\inf}(P_{K,f} \setminus S) = 0$ (resp. $\delta_{\sup}(P_{K,f} \setminus S) = 0$).
Since any prime in $P_{K,f}$ splits into at most $[L:K]$ primes in $L/K$, 
we have 
$\delta_{\inf}(P_{L,f} \setminus S(L)) \leq [L:K]\delta_{\inf}(P_{K,f} \setminus S) = 0$ 
(resp. $\delta_{\sup}(P_{L,f} \setminus S(L)) \leq [L:K]\delta_{\sup}(P_{K,f} \setminus S) = 0$), and hence 
$\delta_{\inf}(P_{L,f} \setminus S(L)) = 0$ (resp. $\delta_{\sup}(P_{L,f} \setminus S(L)) = 0$).
Again by \cite[Lemma 4.6]{Shimizu}, we obtain $\delta_{\sup}(S(L)) = 1$ (resp. $\delta(S(L)) = 1$).
\end{proof}

\begin{prop}\label{5.4}
If $\sum_{p \in \Sigma(\cC)} p^{-1} = \infty$, then $\delta(\underline{P_{K,f}^{\cC}})=1$ and $\delta(P_{K,f}^{\cC})=1$.
\end{prop}
\begin{proof}
We have 
\begin{equation*}
\begin{split}
P_{K,f}^{\cC} = \cup_{l \in \Sigma(\cC)}P_{K,f}^{\cC,l} \supset \cup_{l \in \Sigma(\cC)}P_{\bQ,f}^{\cC,l}(K) = P_{\bQ,f}^{\cC}(K),
\end{split}
\end{equation*}
so that 
$\underline{P_{K,f}^{\cC}} \supset P_{\bQ,f}^{\cC}$.
Further, 
\begin{equation*}
\begin{split}
P_{\bQ,f}^{\cC} \supset 
\cup_{l \in \Sigma(\cC)\setminus\{ 2 \}}P_{\bQ,f}^{\cC,l}
\supset 
\cup_{l \in \Sigma(\cC)\setminus\{ 2 \}}\cs(\bQ(\mu_l)/\bQ).
\end{split}
\end{equation*}
Thus, the assertions follow from 
Proposition \ref{5.2} and Lemma \ref{5.4.5}.
\end{proof}

\section{The main results}
In this section, we prove the main reconstruction results and their consequences. In addition, we conclude with two conjectures concerning reconstruction without a density assumption.

\begin{theorem}\label{6.1}
For $i=1,2$, let $K_i$ be a number field, and $\cC_i$ a nontrivial full class of finite groups.
Assume that $\delta_{\sup}(\Sigma(\cC_i)) > 0$ for one $i$.
Then 
the canonical map $\Iso(K_2^{\cC_2}/K_2, K_1^{\cC_1}/K_1) \to \Iso(G_{K_1}^{\cC_1}, G_{K_2}^{\cC_2})$ is bijective.
In particular, if $G_{K_1}^{\cC_1} \simeq G_{K_2}^{\cC_2}$, then $K_1 \simeq K_2$.
\end{theorem}
\begin{proof}
If $\Iso(G_{K_1}^{\cC_1}, G_{K_2}^{\cC_2}) = \emptyset$, 
then $\Iso(K_2^{\cC_2}/K_2, K_1^{\cC_1}/K_1) = \emptyset$, 
so that the assertion is trivial.
Assume $\Iso(G_{K_1}^{\cC_1}, G_{K_2}^{\cC_2}) \not= \emptyset$. Let $\sigma :G_{K_1}^{\cC_1}\isom G_{K_2}^{\cC_2}$ be any isomorphism.
It suffices to show that there exists a unique $\tau \in \Iso(K_2^{\cC_2}/K_2, K_1^{\cC_1}/K_1)$ such that $\sigma$ coincides with the isomorphism induced by $\tau$.
Write $\Sigma \defeq \Sigma(\cC_1) (= \Sigma(\cC_2)$ by Lemma \ref{2.0}) and $S_0 \defeq \underline{P_{K_1,f}^{\cC_1}} \cap \cs(K_1/\bQ) (= \underline{P_{K_2,f}^{\cC_2}} \cap \cs(K_2/\bQ)$ by Proposition \ref{incomm5} $(ii)$).
Let us 
prove that the conditions in Proposition \ref{U4} hold.
(a), (c) in Proposition \ref{U4} follow from the assumption, Proposition \ref{incomm5} $(iii)$, respectively.
Since 
$$\cs(K_1/\bQ) \setminus S_0 =\cs(K_1/\bQ) \cap (P_{\bQ,f} \setminus \underline{P_{K_1,f}^{\cC_1}}) \subset P_{\bQ,f} \setminus \underline{P_{K_1,f}^{\cC_1}},$$
we have 
$$\delta_{\sup}(\cs(K_1/\bQ) \setminus S_0) \leq \delta_{\sup}(P_{\bQ,f} \setminus \underline{P_{K_1,f}^{\cC_1}}) =0,$$
where the equality follows from \cite[Lemma 4.6]{Shimizu}, Lemma \ref{incomm3} $(i)$ and Proposition \ref{5.4}.
Thus, $\delta(\cs(K_1/\bQ) \setminus S_0) =0$, and hence (b) in Proposition \ref{U4} holds.
\end{proof}

\begin{rem}\label{6.4}
$(i)$
For $i=1,2$, let $K_i$ be a number field,  and $\cC_i$ a nontrivial full class of finite groups.
In this paper so far 
the question of whether 
the existence of an isomorphism between $G_{K_1}^{\cC_1}$ and $G_{K_2}^{\cC_2}$ 
implies 
$\cC_1 = \cC_2$
is not yet answered, 
even if it is induced by a field isomorphism.
This question is completely group-theoretical.

\noindent
$(ii)$
A full class $\cC$ is completely determined by $\Sigma(\cC)$ and the subclass $\cC'$ of $\cC$ consisting of all non-abelian finite simple groups in $\cC$.
Note that if $2 \not\in\Sigma(\cC)$, then $\cC'$ is empty by the Feit-Thompson theorem (cf. \cite[THEOREM]{Feit-Thompson}).
Since $\Sigma(\cC)$ can be recovered group-theoretically from $G_{K}^{\cC}$ by Lemma \ref{2.0}, 
the question in $(i)$ is positively answered 
if 
(for every full class $\cC$ and every number field $K$,)
every finite group in $\cC'$ is isomorphic to some subquotient of $G_{K}^{\cC}$.
To prove 
the latter condition, 
for example, 
it is sufficient 
to solve the inverse Galois problem for every finite group in $\cC'$ and $K$ (or a 
subextension of $K^{\cC}/K$), 
or to show the existence of 
a subquotient of $G_{K}^{\cC}$ which is 
a free pro-$\cC$ group of rank $\geq 2$ (cf. \cite[Proposition 17.6.2]{Fried-Jarden}).
\end{rem}

Theorem \ref{6.1} can be generalized to the following, 
which is also a generalization of \cite[THEOREM]{Uchida2}.

\begin{cor}\label{solv.cl.}
Let $\Sigma \subset P_{\bQ,f}$.
For $i=1,2$, let $K_i$ be a number field, and $\Omega_i$ a Galois extension of $K_i$ having no nontrivial abelian pro-$(\Sigma)$ extension.
Assume that $\delta_{\sup}(\Sigma) > 0$.
Then 
the canonical map $\Iso(\Omega_2/K_2, \Omega_1/K_1) \to \Iso(G(\Omega_1/ K_1), G(\Omega_2/ K_2))$ is bijective.
\end{cor}
\begin{proof}
If $\Iso(G(\Omega_1/ K_1), G(\Omega_2/ K_2)) = \emptyset$, 
then $\Iso(\Omega_2/K_2, \Omega_1/K_1) = \emptyset$, 
so that the assertion is trivial.
Assume $\Iso(G(\Omega_1/ K_1), G(\Omega_2/ K_2)) \not= \emptyset$. Let $\sigma :G(\Omega_1/ K_1)\isom G(\Omega_2/ K_2)$ be any isomorphism.
It suffices to show that there exists a unique $\tau \in \Iso(\Omega_2/K_2, \Omega_1/K_1)$ such that $\sigma$ coincides with the isomorphism induced by $\tau$.
Take any open normal subgroup $U_1$ 
of $G(\Omega_1/ K_1)$.
Set $U_2 \defeq \sigma(U_1)$.
For $i=1,2$, write $L_i$ for the finite Galois subextension of $\Omega_i/K_i$ corresponding to $U_i$.
For $i=1,2$, and for any open normal subgroup $U_i'$ of $G(\Omega_i/ K_i)$ containing $U_i$ with $\sigma(U_1') = U_2'$, write $L_i'$ for the finite Galois subextension of $L_i/K_i$ corresponding to $U_i'$.
We set 
\begin{equation*}
\mathfrak{A}_{U_1}\defeq 
\left\{ \tau \in G_\bQ \left|
\begin{array}{l}
\text{
for any $U_1'$, $U_2'$
as above, 
$\tau|_{L_2'} \in \Iso(L_2'/K_2, L_1'/K_1)$
and the isomorphism}\\ 
\text{$G(L_1'/K_1) \isom G(L_2'/K_2)$ induced by $\sigma$ 
coincides with the isomorphism}\\ 
\text{induced by $\tau|_{L_2'}$
}
\end{array}
\right.\right\}.
\end{equation*}
Note that $\mathfrak{A}_{U_1}$ is a closed subset of $G_\bQ$.
In order to prove 
the existence of $\tau$, 
it suffices to show that 
$\mathfrak{A}_{U_1} \neq \emptyset$ 
for any $U_1$.
Indeed, having shown this, 
we obtain $\cap_{U_1} \mathfrak{A}_{U_1} \neq \emptyset$ by the compactness of $G_\bQ$, and the restriction to $\Omega_2$ of any isomorphism in $\cap_{U_1} \mathfrak{A}_{U_1}$ induces $\sigma$.

Let $\cC$ be the full class of all finite solvable groups in $(\Sigma)$. Then $\Sigma(\cC) = \Sigma$.
Write $\sigma|_{U_1}^\cC \colon U_1^\cC\isom U_2^\cC$ for the maximal pro-$\cC$ quotient of $\sigma|_{U_1} \colon U_1\isom U_2$.
Since $\Omega_i$ has no nontrivial 
pro-$\cC$ 
extension, we have $U_i^\cC = G_{L_i}^{\cC}$.
Therefore, by the assumption and Theorem \ref{6.1}, there exists $\tau' \in G_\bQ$ 
such that 
$\tau'|_{L_2^{\cC}} \in\Iso(L_2^{\cC}/L_2, L_1^{\cC}/L_1)$
and 
$\sigma|_{U_1}^\cC$ coincides with the isomorphism induced by $\tau'|_{L_2^{\cC}}$.
Write $(\tau'|_{L_2^{\cC}})^\ast\colon \Aut(L_1^{\cC}) \isom \Aut(L_2^{\cC})$ for the isomorphism induced by $\tau'|_{L_2^{\cC}}$.
Then $\sigma|_{U_1}^\cC = (\tau'|_{L_2^{\cC}})^\ast|_{U_1^\cC}$.
By Proposition \ref{1.5}, 
the conjugation action 
on $U_2^\cC$ 
of the elements in $\Aut(L_2^{\cC})$ 
mapping $L_2$ to $L_2$ 
is faithful. 
Write $\overline{\sigma} \colon G(L_1^{\cC}/K_1) \isom G(L_2^{\cC}/K_2)$ for the isomorphism induced by $\sigma$.
For any $g_1 \in G(L_1^{\cC}/K_1)$ and any $u_2 \in U_2^\cC$, we have 
$$\overline{\sigma}(g_1) u_2 \overline{\sigma}(g_1)^{-1} 
= (\tau'|_{L_2^{\cC}})^\ast(g_1) u_2 (\tau'|_{L_2^{\cC}})^\ast(g_1)^{-1}.$$
Therefore, we obtain $\overline{\sigma} = (\tau'|_{L_2^{\cC}})^\ast|_{G(L_1^{\cC}/K_1)}$.
Hence 
$\tau'|_{L_2^{\cC}}$ is compatible with the actions of 
$G(L_2^{\cC}/K_2)$ and $G(L_1^{\cC}/K_1)$ 
on $L_2^{\cC}$ and $L_1^{\cC}$.
Thus, 
for any $U_1'$, $U_2'$ as above, 
$\tau'|_{L_2^{\cC}}$ induces an isomorphism between $(L_2^{\cC})^{G(L_2^{\cC}/L_2')} = L_2'$ and $(L_1^{\cC})^{G(L_1^{\cC}/L_1')} = L_1'$.
Therefore, we obtain 
$\tau' \in \mathfrak{A}_{U_1}$.
Further, by the faithfulness as above, $\tau'|_{L_2^{\cC}}$ is uniquely determined by the property $\overline{\sigma} = (\tau'|_{L_2^{\cC}})^\ast|_{G(L_1^{\cC}/K_1)}$.
The uniqueness of $\tau$ follows from this.
\end{proof}

\begin{rem}\label{rem.solv.cl.}
Let $\Sigma \subset P_{\bQ,f}$, and $\Omega$ a Galois extension of $K$ having no nontrivial abelian pro-$(\Sigma)$ extension.
Just as we extended Theorem \ref{6.1} to Corollary \ref{solv.cl.}, 
most results in this paper could 
be extended to certain generalizations where we replace 
$G_{K}^{\cC}$ by $G(\Omega/K)$, $\Sigma(\cC)$ by $\Sigma$ and so on, 
by appropriately modifying the definitions of the symbols and the proofs.\footnote{
  The author does not know whether the maximal set $\Sigma\subset P_{\bQ,f}$ such that $\Omega$ has no nontrivial abelian pro-$(\Sigma)$ extension can be recovered group-theoretically from $G(\Omega/K)$. This would generalize Lemma \ref{2.0}.
  }
The proof of Corollary \ref{solv.cl.} is reduced to Theorem \ref{6.1}; however, 
by 
a similar argument to the proof of Theorem \ref{6.1} 
where we use 
these generalized results instead, we could prove directly Corollary \ref{solv.cl.} without using Theorem \ref{6.1}.
\end{rem}

\begin{cor}\label{6.6}
Notations and assumptions are the same as in Theorem \ref{6.1}.
Further, assume that 
either $\cC_1 = \cC_2$ or 
$\Iso(G_{K_1}^{\cC_1}, G_{K_2}^{\cC_2}) \not= \emptyset$.
Then 
there exists a canonical bijection 
$\Iso(K_2, K_1) \to \OutIso(G_{K_1}^{\cC_1}, G_{K_2}^{\cC_2})$.
\end{cor}
\begin{proof}
Assume that $\Iso(G_{K_1}^{\cC_1}, G_{K_2}^{\cC_2}) \not= \emptyset$.
By Theorem \ref{6.1}, we have $\Iso(K_2^{\cC_2}/K_2, K_1^{\cC_1}/K_1) \not= \emptyset$.
$G_{K_2}^{\cC_2}$ acts on $\Iso(K_2^{\cC_2}/K_2, K_1^{\cC_1}/K_1)$ by the rule $\sigma(\phi) \defeq \phi\circ\sigma^{-1}$.
Then we have a bijection 
$\Iso(K_2^{\cC_2}/K_2, K_1^{\cC_1}/K_1)/G_{K_2}^{\cC_2} \simeq \Iso(K_2, K_1)$.
By Theorem \ref{6.1}, we have a bijection 
$\Iso(K_2^{\cC_2}/K_2, K_1^{\cC_1}/K_1) \to \Iso(G_{K_1}^{\cC_1}, G_{K_2}^{\cC_2})$, which is easily seen to be $G_{K_2}^{\cC_2}$-equivariant if we let $G_{K_2}^{\cC_2}$ act by inner automorphisms on the right-hand side (cf. Notation).
Thus, factoring out by the $G_{K_2}^{\cC_2}$-actions,
we obtain the required bijection.

Assume that $\cC_1 = \cC_2$
 and $\Iso(G_{K_1}^{\cC_1}, G_{K_2}^{\cC_2}) = \emptyset$.
If $\Iso(K_2, K_1) \not= \emptyset$, then $\Iso(K_2^{\cC_2}/K_2, K_1^{\cC_1}/K_1) \not= \emptyset$, so that $\Iso(G_{K_1}^{\cC_1}, G_{K_2}^{\cC_2}) \not= \emptyset$, 
a contradiction. Thus, $\Iso(K_2, K_1) = \emptyset$. Then the assertion is trivial.
\end{proof}

\begin{cor}\label{6.7}
Assume that $\delta_{\sup}(\Sigma(\cC)) > 0$.
Then 
there is a canonical isomorphism 
$\Aut(K) \isom \Out(G_{K}^{\cC})$.
Further, assume that $\Aut(K)$ is trivial.
Then the canonical homomorphism 
$G_{K}^{\cC} \to \Aut(G_{K}^{\cC})$ induced by the conjugation action is an isomorphism, 
in particular, 
all automorphisms of $G_{K}^{\cC}$ are inner.
\end{cor}
\begin{proof}
The first assertion follows immediately from Corollary \ref{6.6}.
Therefore, 
if $\Aut(K)$ is trivial, then $\Inn(G_{K}^{\cC})=\Aut(G_{K}^{\cC})$.
By Proposition \ref{1.6}, the canonical surjection 
$G_{K}^{\cC} \to \Inn(G_{K}^{\cC})$ is an isomorphism.
Thus, the second assertion follows.
\end{proof}

We next establish two propositions that will be used to prove the relative reconstruction theorem (Theorem \ref{relGC}).

\begin{prop}\label{charthm}
For $i=1,2$, let $K_i$ be a number field, $\cC_i$ a nontrivial full class of finite groups, 
$S_i \subset P_{K_i,f}$, 
and $\sigma :G_{K_1}^{\cC_1}\isom G_{K_2}^{\cC_2}$ an isomorphism.
Assume that the following conditions hold:
\begin{itemize}

\item[(a)]
For one $i$, 
$2 \in\Sigma(\cC_i)$ 
or $K_i$ has a complex prime.

\item[(b)]
There exists a 
local correspondence between $S_1$ and $S_2$ for $\sigma$, 
satisfying condition $(\Char)$.

\item[(c)]
For any finite Galois subextension $L_1$ of $K_1^{\cC_1}/K_1$, 
$\delta_{\sup}(\underline{S_1} \cap \cs(L_1/\bQ)) > 0$.

\item[(d)]
$\underline{S_2} \cap \cs(K_2/\bQ) \not= \emptyset$
.
\end{itemize}
Then 
there exists a unique $\tau \in \Iso(K_2^{\cC_2}/K_2, K_1^{\cC_1}/K_1)$ such that $\sigma$ coincides with the isomorphism induced by $\tau$. 
\end{prop}
\begin{proof}
The uniqueness of $\tau$ follows from Proposition \ref{1.5}. Let us prove the existence of $\tau$.
Write $\Sigma \defeq \Sigma(\cC_1) (= \Sigma(\cC_2)$ by Lemma \ref{2.0}).
Take any open normal subgroup $U_1$ 
of $G_{K_1}^{\cC_1}$.
Set $U_2 \defeq \sigma(U_1)$.
For $i=1,2$, write $L_i$ for the finite Galois subextension of $K_i^{\cC_i}/K_i$ corresponding to $U_i$.
For $i=1,2$, and for any open normal subgroup $U_i'$ of $G_{K_i}^{\cC_i}$ containing $U_i$ with $\sigma(U_1') = U_2'$, write $L_i'$ for the finite Galois subextension of $L_i/K_i$ corresponding to $U_i'$.
We set 
\begin{equation*}
\mathfrak{A}_{U_1}\defeq 
\left\{ \tau \in G_\bQ \left|
\begin{array}{l}
\text{
for any $U_1'$, $U_2'$
as above, 
$\tau|_{L_2'} \in \Iso(L_2'/K_2, L_1'/K_1)$
and the isomorphism}\\ 
\text{$G(L_1'/K_1) \isom G(L_2'/K_2)$ induced by $\sigma$ 
coincides with the isomorphism}\\ 
\text{induced by $\tau|_{L_2'}$
}
\end{array}
\right.\right\}.
\end{equation*}
Note that $\mathfrak{A}_{U_1}$ is a closed subset of $G_\bQ$.
In order to prove 
the existence of $\tau$, 
it suffices to show that 
$\mathfrak{A}_{U_1} \neq \emptyset$ 
for any $U_1$.
Indeed, having shown this, 
we obtain $\cap_{U_1} \mathfrak{A}_{U_1} \neq \emptyset$ by the compactness of $G_\bQ$, and the restriction to  $K_2^{\cC_2}$ of any isomorphism in $\cap_{U_1} \mathfrak{A}_{U_1}$ induces $\sigma$.
Moreover, by making $U_1$ be contained in the open subgroup of $G_{K_1}^{\cC_1}$ corresponding to $K_1(\sqrt{-1})$ if $2 \in\Sigma$, 
or by (a) otherwise, we may assume that $L_i$ has a complex prime for one $i$.

Take $l \in \Sigma$. Set $L \defeq \widetilde{L_1}\widetilde{L_2}$ and 
$S_i' \defeq ((\underline{S_i} \cap \cs(K_i/\bQ)) \setminus (\{ l \} \cup \Ram(L_1L_2/\bQ) ))(L_i)$ for $i=1,2$.
Let us 
prove that $\sigma :G_{K_1}^{\cC_1}\isom G_{K_2}^{\cC_2}$, open normal subgroups $U_1 \subset G_{K_1}^{\cC_1}$, $U_2\subset G_{K_2}^{\cC_2}$, 
and subsets $S_1' \subset P_{L_1,f}$, $S_2' \subset P_{L_2,f}$
satisfy the conditions in Proposition \ref{4.7}. 
Conditions (a), (d) in Proposition \ref{4.7} hold clearly.
Condition (b) in Proposition \ref{4.7} follows from 
(b), (c), (d) in the assumption and Lemma \ref{CharDeg} $(iii)$.
By Lemma \ref{CharDeg} $(i)$, we have 
$$\underline{S_1} \cap \cs(L/\bQ) = \underline{S_1} \cap \cs(L_1/\bQ) \cap \cs(L_2/\bQ) = \underline{S_2} \cap \cs(L_2/\bQ) = \underline{S_1} \cap \cs(L_1/\bQ).$$ 
Therefore, 
\begin{equation*}
\begin{split} 
\delta_{\sup}(S_1'(L)) 
&= \delta_{\sup}((\underline{S_1} \cap \cs(K_1/\bQ))(L))\\
&\geq \delta_{\sup}((\underline{S_1} \cap \cs(L/\bQ))(L))\\
&= [L:\bQ]\delta_{\sup}(\underline{S_1} \cap \cs(L/\bQ))\\
&= [L:\bQ]\delta_{\sup}(\underline{S_1} \cap \cs(L_1/\bQ))\\
&>0,
\end{split}
\end{equation*}
where the second equality and the last inequality 
follow from \cite[Lemma 4.1]{Shimizu} and (c) in the assumption, 
respectively.
Thus, condition (c) in Proposition \ref{4.7} holds.
By Proposition \ref{4.7}, 
there exists $\tau \in G(L/\bQ)$ such that 
$\tau|_{L_2} \in \Iso(L_2/K_2, L_1/K_1)$ and the isomorphism $G(L_1/K_1) \isom G(L_2/K_2)$ induced by $\sigma$ coincides with the isomorphism induced by $\tau|_{L_2}$, 
and in particular, 
by taking $\widetilde{\tau} \in G_\bQ$ with $\widetilde{\tau}|_{L}=\tau$, we obtain 
$\widetilde{\tau} \in \mathfrak{A}_{U_1}$.
\end{proof}

\begin{rem}\label{charrem}
We use the notations in Proposition \ref{charthm}.
Write $\Sigma \defeq \Sigma(\cC_1) (= \Sigma(\cC_2)$ by Lemma \ref{2.0}).
If $\delta_{\sup}(\Sigma) > 0$, then, by Theorem \ref{6.1}, there exists a unique $\tau \in \Iso(K_2^{\cC_2}/K_2, K_1^{\cC_1}/K_1)$ such that $\sigma$ coincides with the isomorphism induced by $\tau$. 
Let us consider the case 
that $\delta(\Sigma) = 0$ and that condition (a) in Proposition \ref{charthm} holds.
In this case, 
the Dirichlet density of the set $
P_{L_1,f}^{\cC_1} \cap \Sigma(L_1)$ in Theorem \ref{3.5} 
is $0$, 
so that 
if we set $S_i \defeq P_{K_i,f}^{\cC_i} \cap \Sigma(K_i)$ for $i=1,2$, then 
condition (c) in Proposition \ref{charthm} does not hold.
Thus, to apply Proposition \ref{charthm} when $\delta(\Sigma)=0$, one needs a local correspondence satisfying $(\Char)$ on sets fulfilling the density condition (c).
The proof of the next proposition obtains such a correspondence from compatibility with suitable cyclotomic quotients.

\end{rem}

\begin{prop}\label{relthm}
For $i=1,2$, let $K_i$ be a number field, $\cC_i$ a nontrivial full class of finite groups, 
$\Sigma' \subset P_{\bQ,f}$, $l_0 \in 
\Sigma(\cC_1)\cap\Sigma(\cC_2)
$, 
and $\sigma :G_{K_1}^{\cC_1}\isom G_{K_2}^{\cC_2}$ an isomorphism.
Assume that the following conditions hold:
\begin{itemize}

\item[(a)]
For one $i$, 
$2 \in\Sigma(\cC_i)$ 
or $K_i$ has a complex prime.

\item[(b)]
$\delta_{\sup}(\Sigma'_{l_0}) > 0$.

\item[(c)]
$\sigma$ induces an isomorphism 
\begin{multline*}
\Ker\bigl(
G_{K_1}^{\cC_1}\twoheadrightarrow G_{K_1}^{(l_0)}
\overset{\chi_{K_1,l}^{(l_0)}}\longrightarrow
{\bZ_l}^{\ast,(l_0)}
\twoheadrightarrow
{\bZ_l}^{\ast,(l_0)}/{\bZ_l}^{\ast,(l_0),l_0}
\bigr)
\\
\isom
\Ker\bigl(
G_{K_2}^{\cC_2}\twoheadrightarrow G_{K_2}^{(l_0)}
\overset{\chi_{K_2,l}^{(l_0)}}\longrightarrow
{\bZ_l}^{\ast,(l_0)}
\twoheadrightarrow
{\bZ_l}^{\ast,(l_0)}/{\bZ_l}^{\ast,(l_0),l_0}
\bigr)
\end{multline*}
for each $l \in \Sigma'_{l_0}$.

\end{itemize}
Then 
there exists a unique $\tau \in \Iso(K_2^{\cC_2}/K_2, K_1^{\cC_1}/K_1)$ such that $\sigma$ coincides with the isomorphism induced by $\tau$. 
\end{prop}
\begin{proof}
Removing a finite subset from $\Sigma'$, we may assume that $K_i/\bQ$ and $\bQ_{\Sigma',l_0}/\bQ$ are linearly disjoint for each $i \in \{ 1,2 \}$.
Then we have $G(K_{i,\Sigma',l_0}/K_i) \simeq G(\bQ_{\Sigma',l_0}/\bQ)$.
By (c), 
for each $l \in \Sigma'_{l_0}$, there exists $a_l \in \bF_{l_0}^{\ast}$ such that the following diagram commutes: 
\begin{equation*}
\xymatrix{
G_{K_1}^{\cC_1}\ar[rr]^-{\simeq}_{\sigma} \ar@{->>}[d]&&G_{K_2}^{\cC_2}\ar@{->>}[d]\\
G(K_1(\mu_{l})/K_1)/G(K_1(\mu_{l})/K_1)^{l_0}\ar[rr]^-{\simeq}\ar[d]^-{\simeq}&&G(K_2(\mu_{l})/K_2)/G(K_2(\mu_{l})/K_2)^{l_0}\ar[d]^-{\simeq}\\
G(\bQ(\mu_{l})/\bQ)/G(\bQ(\mu_{l})/\bQ)^{l_0}\ar[rr]^-{\simeq}_{(\ )^{a_l}}&&G(\bQ(\mu_{l})/\bQ)/G(\bQ(\mu_{l})/\bQ)^{l_0}
}
\end{equation*}
where the bottom horizontal arrow is the automorphism of $G(\bQ(\mu_{l})/\bQ)/G(\bQ(\mu_{l})/\bQ)^{l_0}(\simeq \bF_{l_0})$ obtained by taking $a_l$-th powers.
By (b) 
and the finiteness of $\# \bF_{l_0}^{\ast}$, there exists $a \in \bF_{l_0}^{\ast}$ such that 
\begin{equation*}
\delta_{\sup}(
\{ l \in \Sigma'_{l_0} \mid a_l = a \}
) > 0.
\end{equation*}
Replacing $\Sigma'$ by $\{ l \in \Sigma'_{l_0} \mid a_l = a \}$, 
we may assume that 
the following diagram commutes: 
\begin{equation*}
\xymatrix{
G_{K_1}^{\cC_1}\ar[rr]^-{\simeq}_{\sigma} \ar@{->>}[d]&&G_{K_2}^{\cC_2}\ar@{->>}[d]\\
G(K_{1,\Sigma',l_0}/K_1)\ar[rr]^-{\simeq}\ar[d]^-{\simeq}&&G(K_{2,\Sigma',l_0}/K_2)\ar[d]^-{\simeq}\\
G(\bQ_{\Sigma',l_0}/\bQ)\ar[rr]^-{\simeq}_{(\ )^{a}}&&G(\bQ_{\Sigma',l_0}/\bQ)
}
\end{equation*}
Write $\Sigma \defeq \Sigma(\cC_1) (= \Sigma(\cC_2)$ by Lemma \ref{2.0}).
Set $S_i \defeq A_i \setminus T_i
$
for $i=1,2$, 
where 
a subset $A_i \subset P_{K_i,f}^{\cC_i}$ 
and a finite subset $T_i \subset P_{K_i,f}^{\cC_i}$
are taken 
as in the assertion of Proposition \ref{relloccor}.
Let us prove that the conditions in Proposition \ref{charthm} hold.
Condition (a) in Proposition \ref{charthm} follows from 
(a) in the assumption.
By (b) in the assumption and Proposition \ref{relloccor} $(i)$, 
condition (b) in Proposition \ref{charthm} holds.
Let $l \in \Sigma$.
By Remark \ref{csset} $(ii)$, $P_{K_i,f}^{\cC_i,l} \supset (\cs(\bQ(\mu_{l})/\bQ)\setminus \{ l \})(K_i)$, so that 
$\underline{P_{K_i,f}^{\cC_i}} \supset \underline{P_{K_i,f}^{\cC_i,l}} \supset \cs(\bQ(\mu_{l})/\bQ)\setminus \{ l \}$.
Therefore, 
for any finite Galois subextension $L_i$ of $K_i^{\cC_i}/K_i$, 
\begin{equation*}
\begin{split} 
\delta_{\sup}(\underline{S_i} \cap \cs(L_i/\bQ)) 
&= \delta_{\sup}(\underline{P_{K_i,f}^{\cC_i} \setminus T_i}  \cap \cs(L_i/\bQ))\\
&= \delta_{\sup}(\underline{P_{K_i,f}^{\cC_i}} \cap \cs(L_i/\bQ))\\
&\geq \delta_{\sup}(\cs(\bQ(\mu_{l})/\bQ) \cap\cs(L_i/\bQ))\\
&=\delta_{\sup}(\cs(\widetilde{L_i}(\mu_{l})/\bQ))\\
&= \frac{1}{[\widetilde{L_i}(\mu_l):\bQ]}\\
&>0,
\end{split}
\end{equation*}
where the first and last 
equalities follow from Proposition \ref{relloccor} $(ii)$ and the Chebotarev density theorem, respectively.
Thus, conditions (c), (d) in Proposition \ref{charthm} hold.
\end{proof}

\begin{rem}\label{relrem}
We use the notations in Proposition \ref{relthm}.
Condition (c) in Proposition \ref{relthm} holds 
if, for every $l \in \Sigma'_{l_0}$, the following diagram commutes:
\begin{equation*}
\xymatrix{
G_{K_1}^{\cC_1}\ar[rr]^-{\simeq}_{\sigma} \ar@{->>}[d]&&G_{K_2}^{\cC_2}\ar@{->>}[d]\\
G_{K_1}^{(l_0)}\ar[rr]^-{\simeq}_{} \ar[rd]_-{}_{\chi_{K_1,l}^{(l_0)}}&&G_{K_2}^{(l_0)}\ar[dl]^-{\chi_{K_2,l}^{(l_0)}}\\
&{\bZ_l}^{\ast, (l_0)}.
}
\end{equation*}

\end{rem}

Proposition \ref{relthm} yields the following relative reconstruction theorem, which is an analogue of the relative Grothendieck conjecture.
\begin{theorem}\label{relGC}
For $i=1,2$, let $K_i$ be a number field, $\cC_i$ a nontrivial full class of finite groups, 
and $\sigma :G_{K_1}^{\cC_1}\isom G_{K_2}^{\cC_2}$ an isomorphism.
Write $\Sigma \defeq \Sigma(\cC_1) (= \Sigma(\cC_2)$ by Lemma \ref{2.0}).
Assume that the following conditions hold:
\begin{itemize}

\item[(a)]
$2\in\Sigma$ or at least one of $K_1,K_2$ has a complex prime.

\item[(b)]
The following diagram commutes:
\begin{equation*}
\xymatrix{
G_{K_1}^{\cC_1}\ar[rr]^-{\simeq}_{\sigma} \ar[rd]_-{}&&G_{K_2}^{\cC_2}\ar[dl]^-{}\\
&G_\bQ^{\ab,\Sigma},
}
\end{equation*}
where the diagonal arrows are the restrictions 
and $G_\bQ^{\ab,\Sigma} = G(\bQ(\cup_{n \in \bZ_{>0}} \mu_{n})/\bQ)^{\Sigma}$ by the Kronecker-Weber theorem.

\end{itemize}
Then 
there exists a unique $\tau \in \Iso(K_2^{\cC_2}/K_2, K_1^{\cC_1}/K_1)$ such that $\sigma$ coincides with the isomorphism induced by $\tau$. 
\end{theorem}
\begin{proof}
Let $l_0 \in \Sigma$ and $\Sigma' \defeq P_{\bQ,f}$.
Then $\Sigma'_{l_0} = \cs(\bQ(\mu_{l_0})/\bQ) \setminus \{ l_0 \}$, and hence the conditions in Proposition \ref{relthm} follow from the assumption.
\end{proof}

The goal of the rest of this section is to reduce Conjecture \ref{7.1} 
to Conjecture \ref{7.2}.
\begin{conj}\label{7.1}
For $i=1,2$, let $K_i$ be a number field, $\cC_i$ a nontrivial full class of finite groups, and $\sigma :G_{K_1}^{\cC_1}\isom G_{K_2}^{\cC_2}$ an isomorphism.
Assume that the following condition holds:
\begin{itemize}

\item[(a)]
For one $i$, $2 \in\Sigma(\cC_i)$ or $K_i$ has a complex prime.

\end{itemize}
Then 
there exists a unique $\tau \in \Iso(K_2^{\cC_2}/K_2, K_1^{\cC_1}/K_1)$ such that $\sigma$ coincides with the isomorphism induced by $\tau$.
\end{conj}

\begin{conj}\label{7.2}
We use the notations in Conjecture \ref{7.1}.
By Theorem \ref{3.5}, there exists a local correspondence $\phi$ between $P_{K_1,f}^{\cC_1}$ and $P_{K_2,f}^{\cC_2}$ for $\sigma$.
Then 
$\phi$ satisfies condition $(\Char)$.
\end{conj}

\begin{prop}\label{7.4}
If Conjecture \ref{7.2} holds, 
then Conjecture \ref{7.1} holds.
\end{prop}
\begin{proof}
Assume Conjecture \ref{7.2} holds.
Set $S_i \defeq P_{K_i,f}^{\cC_i}$
for $i=1,2$.
Then condition (b) in Proposition \ref{charthm} holds.
By a similar argument to the proof of Proposition \ref{relthm}, 
we have 
\begin{equation*}
\begin{split} 
\delta_{\sup}(\underline{S_i} \cap \cs(L_i/\bQ)) 
\geq \delta_{\sup}(\cs(\bQ(\mu_{l})/\bQ) \cap\cs(L_i/\bQ))
= \frac{1}{[\widetilde{L_i}(\mu_l):\bQ]}
>0
\end{split}
\end{equation*}
for any $l \in \Sigma(\cC_1)$ and any finite Galois subextension $L_i$ of $K_i^{\cC_i}/K_i$, and hence 
conditions (c), (d) in Proposition \ref{charthm} hold.
\end{proof}

\begin{rem}\label{7.5}
We use the notations in Conjecture \ref{7.1}.
Write $\Sigma \defeq \Sigma(\cC_1) (= \Sigma(\cC_2)$ by Lemma \ref{2.0}).

\noindent
$(i)$
Conjecture \ref{7.1} states that as long as condition (a) holds, 
any isomorphism between $G_{K_1}^{\cC_1}$ and $G_{K_2}^{\cC_2}$ comes from a unique isomorphism between $K_2^{\cC_2}/K_2$ and $K_1^{\cC_1}/K_1$, 
even in the case when $\Sigma = \{ p \}$, that is,  
$G_{K_i}^{\cC_i} = G_{K_i}^{(p)}$ for $i=1,2$.

\noindent
$(ii)$
Let $l \in \Sigma$. By Theorem \ref{2.9} and Proposition \ref{1.4}, there exists a 
local correspondence $\phi^l$ between $P_{K_1,f}^{\cC_1,l}$ and $P_{K_2,f}^{\cC_2,l}$ for $\sigma$.
By an easy modification of the proof of Proposition \ref{7.4}, 
we can replace $\phi$ in Proposition \ref{7.4} by $\phi^l$ (for some $l \in \Sigma$).
Therefore, 
to prove Conjecture \ref{7.1}, 
it suffices to show, 
instead of Conjecture \ref{7.2}, 
the weaker condition 
that 
$\phi^l$ satisfies condition $(\Char)$ (for some $l \in \Sigma$).

\noindent
$(iii)$
For each prime $p$, Koymans and Pagano construct isomorphisms between the maximal two-step nilpotent pro-$p$ quotients of the absolute Galois groups of certain non-isomorphic imaginary quadratic fields by ``shuffling places'' (see \cite[Theorem 1.2 and \S 1.5]{Koymans-Pagano}).
They conjecture that the maximal pro-$p$ Galois groups of these fields are also isomorphic (see \cite[Conjecture 1.7]{Koymans-Pagano}).
The conjectured pro-$p$ isomorphisms would give counterexamples to Conjecture \ref{7.1} with $\Sigma=\{p\}$.
By the argument in the proof of Proposition \ref{7.4}, the local correspondences for such isomorphisms (obtained by Theorem \ref{3.5}) would necessarily fail to satisfy condition $(\Char)$, and hence would also give counterexamples to Conjecture \ref{7.2}.

\end{rem}


\begin{thebibliography}{Y}

\bibitem[Feit-Thompson]{Feit-Thompson} Feit, W. and Thompson, John G., Solvability of groups of odd order, Pacific J. Math. 13 (1963), 775--1029.

\bibitem[Fried-Jarden]{Fried-Jarden} Fried, M. D., Jarden, M., Field arithmetic, Third edition, 
Ergebnisse der Mathematik III, vol. 11. Springer-Verlag, Berlin, 2008.

\bibitem[Gras]{Gras} Gras, G., 
Class field theory, From theory to practice,
\textit{Springer Monographs in Mathematics.} 
Springer-Verlag, Berlin, 2003. 

\bibitem[Ivanov]{Ivanov2} Ivanov, A., Arithmetic and anabelian theorems for stable sets in number fields, Dissertation, Universit\"at Heidelberg, 2013.

\bibitem[Ivanov2]{Ivanov} Ivanov, A., On some anabelian properties of arithmetic curves, Manuscripta Mathematica 144 (2014), no. 3, 545--564.

\bibitem[Ivanov3]{Ivanov3} Ivanov, A., On a generalization of the Neukirch-Uchida theorem, Moscow Mathematical J. 17 (2017), no. 3, 371--383.

\bibitem[Koymans-Pagano]{Koymans-Pagano}
Koymans, P. and Pagano, C., Two-step nilpotent extensions are not anabelian, Math. Z. 306 (2024), Paper No. 31.

\bibitem[Mochizuki]{Mochizuki} Mochizuki, S., The absolute anabelian geometry of hyperbolic curves. Galois Theory and Modular Forms, Kluwer Academic Publishers (2003), pp. 77--122.



\bibitem[Montgomery-Vaughan]{Montgomery-Vaughan} Montgomery, H. L. and Vaughan, R. C., The large sieve, Mathematika 20 (1973), 119--134.

\bibitem[Neukirch1]{Neukirch} Neukirch, J., Kennzeichnung der $p$-adischen und der endlichen algebraischen Zahlkörper, Invent. Math. 6 (1969), 296--314.

\bibitem[Neukirch2]{Neukirch2} Neukirch, J., Kennzeichnung der endlich-algebraischen Zahlkörper durch die Galoisgruppe der maximal auflösbaren Erweiterungen, J. Reine Angew. Math. 238 (1969), 135--147.

\bibitem[Neukirch3]{Neukirch3} Neukirch, J., Algebraic Number Theory, Grundlehren der Mathematischen Wissenschaften, 322. Springer-Verlag, Berlin, 1999.

\bibitem[NSW]{NSW}
Neukirch, J., Schmidt, A. and Wingberg, K., 
Cohomology of number fields, Second edition, Grundlehren der Mathematischen Wissenschaften, 323. Springer-Verlag, Berlin, 2008.




\bibitem[\text{Sa\"\i di-Tamagawa}]{Saidi-Tamagawa2}Sa\"\i di, M. and Tamagawa, A., The $m$-step solvable anabelian geometry of number fields, 
Journal f\"ur die reine und angewandte Mathematik (Crelles Journal), vol. 2022, no. 789, 2022, 153--186.

\bibitem[Serre]{Serre2} Serre, J.-P., Local fields, Volume 67 of Graduate Texts in Mathematics, Springer-Verlag, New York, 1979. Translated from the French by Marvin Jay Greenberg.

\bibitem[Serre2]{Serre} Serre, J.-P., Abelian $l$-adic representations and elliptic curves, Second edition, Advanced Book Classics, Addison-Wesley, Redwood City, 1989.

\bibitem[Shimizu]{Shimizu} Shimizu, R., The Neukirch-Uchida theorem with restricted ramification, 
Journal f\"ur die reine und angewandte Mathematik (Crelles Journal), vol. 2022, no. 785, 2022, 187--217.


\bibitem[Shimizu2]{Shimizu2} Shimizu, R., Isomorphisms of Galois groups of number fields with restricted ramification, 
Math. Nachr. 296 (2023), no. 7, 3026--3033.

\bibitem[Sueyoshi]{Sueyoshi}
Sueyoshi, Y., A characterization of number fields by $p$-closed extensions, J. Number Theory 28 (1988), no.~2, 145--151.



\bibitem[Uchida]{Uchida} Uchida, K., Isomorphisms of Galois groups, J. Math. Soc. Japan, 28 (4) (1976), 617--620.

\bibitem[Uchida2]{Uchida2} Uchida, K., Isomorphisms of Galois groups of solvably closed Galois extensions, 
\textit{Tohoku Math. J.} \textbf{31} (1979), 359--362.

\end{thebibliography}
\end{document}